\documentclass[11pt]{article}

\usepackage{bbm}
\usepackage{amscd}
\usepackage{amsmath}
\usepackage{amsfonts}
\usepackage{amssymb}
\usepackage{multirow}
\usepackage{easybmat}
\usepackage{mathrsfs}
\usepackage{pgfplots}
\usepackage{setspace}
\usepackage{float}

\usepgfplotslibrary{polar}

\newtheorem{theorem}{\bf Theorem}[section]
\newtheorem{definition}[theorem]{\bf Definition}
\newtheorem{lemma}[theorem]{\bf Lemma}
\newtheorem{prop}[theorem]{\bf Proposition}
\newtheorem{coro}[theorem]{\bf Corollary}

\newtheorem{example}[theorem]{\bf Example}
\newtheorem{remark}[theorem]{\bf Remark}

\newenvironment{proof}{\noindent{\em Proof:}}{\quad \hfill$\Box$\vspace{2ex}}

\def\no{\noindent}


\numberwithin{equation}{section}
\usepackage{enumitem}
\usepackage{tikz-cd}
\usepackage{indentfirst}
\usepackage{titlesec}
\usepackage{multirow}
\usepackage{tikz}
\usepackage[colorlinks, citecolor=blue]{hyperref}

\DeclareMathOperator{\diag}{diag}
\DeclareMathOperator{\antidiag}{anti-diag}

\DeclareMathOperator{\corank}{corank}

\begin{document}
\bibliographystyle{plain}
\begin{center}{\LARGE \bf Structure of minimal 2-spheres of constant curvature in the complex hyperquadric}
\end{center}
\begin{center}
Quo-Shin Chi$^\ast$, Zhenxiao Xie and Yan Xu\\
$^\ast$ Corresponding Author

\end{center}

\bigskip

\no
{\bf Abstract.}
In this paper, the singular-value decomposition theory of complex matrices is explored to study constantly curved 2-spheres minimal in both $\mathbb{C}P^n$ and the hyperquadric of $\mathbb{C}P^n$. The moduli space of all those noncongruent ones is introduced, which can be described by certain complex symmetric matrices modulo an appropriate group action. Using this description, many examples, such as constantly curved holomorphic 2-spheres of higher degree, nonhomogenous minimal 2-spheres of constant curvature, etc., are constructed. Uniqueness is proven for the totally real constantly curved 2-sphere minimal in both the hyperquadric and $\mathbb{C}P^n$.

\no

\no

\bigskip

\section{Introduction}
In differential geometry, the investigation of compact surfaces characterized by curvature properties and variational equations, such as the study of constantly curved minimal 2-spheres in symmetric spaces, is an enduring and important topic. In space forms (real and complex), the structure of these 2-spheres is simple and well known. For example, any minimal 2-sphere of constant curvature in the complex projective space $\mathbb{C}P^n$ belongs to the Veronese sequence, up to a rigid motion (see \cite{Bando-Ohnita,Bolton1988}). The proof was essentially based on the rigidity theorem of holomorphic curves in $\mathbb{C}P^n$ \cite{Calabi}. However, this rigidity does not hold for generic symmetric spaces, among which the Grassmannian is a prototypical example. This phenomenon was first observed by the first named author and Zheng in \cite{Chi-Zheng}, where noncongruent holomorphic 2-spheres of degree $2$ and constant curvature in $G(2,4,\mathbb{C})$ were classified into two families, using the method of moving frames and Cartan's theory of higher order invariants~\cite{Cartan, Jensen}. Since then, there have emerged many works on constantly curved minimal 2-spheres in the Grassmannian (see \cite{JiaoLiS^2inQn,Li-Yu, Peng-Jiao, PengWangXu, Peng-Xu} and the references therein), most of which were devoted to studying constantly curved minimal 2-spheres in the hyperquadric $\mathcal{Q}_{n-1}$ of $\mathbb{C}P^n$ defined by $z_0^2+\cdots+z_n^2=0$
with respect to the homogeneous coordinates of ${\mathbb C}P^n$, where $\mathcal{Q}_{n-1}$, when identified with the oriented real Grassmannian $\widetilde{G}(2,n+1,\mathbb{R})$, can be seen as the next simplest symmetric spaces beyond space forms. On the other hand, quadrics (smooth or singular) play a fundamental role in regard to the complex Grassmannian, since by the Pl\"ucker embedding, any complex Grassmannian $G(2, n+1,\mathbb{C})$ can be realized as the intersection of quadrics in the associated complex projective space (true, in fact, for any variety).

Even in the case of $\mathcal{Q}_{n-1}$, only some special examples and partial classification (e.g. under the condition of homogeneity or lower dimension) have been obtained up to now. Indeed, with the homogeneous  assumption, Peng, Wang and Xu gave a complete classification of minimal 2-spheres in $\mathcal{Q}_{n-1}$ in \cite{PengWangXu}, where they proposed the following.

{\bf Problem 1.} {\em How to construct a nonhomogenous constantly curved minimal $2$-sphere in $\mathcal{Q}_{n-1}$ for $n\geq 4$}?

{\bf Problem 2.} {\em Does there exist a linearly full totally real minimal $2$-sphere in $\mathcal{Q}_{n-1}$ which is also minimal in $\mathbb{C}P^n$}?

Jiao and Li \cite{JiaoLiS^2inQn}, using the constructive method of harmonic sequence given by Bahy-El-Dien and Wood in \cite{Bahy-Wood}, classified all constantly curved minimal 2-spheres with higher isotropic order in $\widetilde{G}(2,n+1,\mathbb{R})\cong\mathcal{Q}_{n-1}$ under the totally unramified assumption. Based on this, a complete classification of all totally unramified minimal 2-spheres of constant curvature in $\widetilde{G}(2,7,\mathbb{R})\cong\mathcal{Q}_{5}$ has recently been obtained by Jiao and Li in \cite{Jiao-Li}. For classifications in $\mathcal{Q}_{2}, \mathcal{Q}_{3}$ and $\mathcal{Q}_{4}$, we refer to \cite{JiaoWangQn,zbMATH06272104,JiaoLiS^2inQn} and the references therein.

One observes that a constantly curved minimal 2-sphere of $\mathcal{Q}_{n-1}$ in all these classifications either is also minimal in $\mathbb{C}P^n$, or can be constructed from a totally real constantly curved 2-sphere both minimal in $\mathcal{Q}_{n-1}$ and $\mathbb{C}P^{n}$; moreover, almost all of them are homogeneous. With this observation, the present paper is contributed to studying constantly curved 2-spheres minimal in both $\mathcal{Q}_{n-1}$ and $\mathbb{C}P^{n}$. By the theory of singular-value decomposition (denoted by SVD in this paper) of complex matrices, a method of constructing such kind of 2-spheres is introduced, from which an abundance of nonhomogenous examples can be constructed to answer {\bf Problem 1}. The existence part in {\bf Problem 2} has been affirmed in the classification results of Jiao and Li \cite{JiaoLiS^2inQn}.  We obtain the uniqueness part as follows; see also Corollary~\ref{cor-uniqueness}.
\begin{theorem}
Suppose a linearly full totally real minimal $2$-sphere of constant curvature $8/(d^2+2d)$ in $\mathcal{Q}_{n-1}$ is also minimal in $\mathbb{C}P^{n}$. Then $d$ is even and $n=2d+1$. Moreover, it is unique up to a real orthogonal transformation.
\end{theorem}

This SVD method, novel in a sense, is effective and unifying in describing the moduli space of noncongruent 2-spheres with the same constant curvature, such that they are minimal in both $\mathcal{Q}_{n-1}$ and $\mathbb{C}P^{n}$;  
see Theorem~\ref{simple classification}.  As an example, the aforementioned result of Chi and Zheng follows from our classification of constantly curved holomorphic $2$-spheres of degree no more than 3, to be done in Section~\ref{sec-classify}. More generally, the classification of all constantly curved  holomorphic $2$-spheres in $G(2,4,{\mathbb C})$ was obtained by Li and Jin  in~\cite{zbMATH05590797} by a direct calculation via elaborate coordinate changes. We will give a systematic SVD proof of it in Proposition~\ref{Pr}.


As another example, recently, using a sophisticated method from the perspective of holomorphic isometric embeddings of $\mathbb{C}P^1$ in $\mathcal{Q}_{n-1}$, Macia, Nagatomo and Takahashi \cite{M-N-T} studied the moduli space of noncongruent constantly curved holomorphic 2-spheres of degree $(n-1)/2$ in $\mathcal{Q}_{n-1}$ when $n$ is odd. The real dimension of this moduli space was determined by them to be $(n^2-4n-1)/4$.
We point out that the dimension count can also be attained via the SVD method and the fact that the ideal of a rational normal curve of degree $d$ is generated by $d^2-d$ independent quadrics. Note that a rational normal curve of degree $d$ lies in $\mathcal{Q}_{n-1}$, if and only if,  the quadric given by the intersection of $\mathcal{Q}_{n-1}$ and the projective $d$-plane spanned by the curve belongs to the ideal of the curve. Conversely, to guarantee that quadrics in this ideal belong to $\mathcal{Q}_{n-1}$, the SVD method reveals that there is no other constraint if $d= (n-1)/2$ when $n$ is odd, whence follows the dimension count (see Theorem~\ref{coro-dim}).

On the other hand, when $d>(n-1)/2$, there are other constraints (see Proposition~\ref{classification of cdn}). This makes the study of constantly curved holomorphic 2-spheres in $\mathcal{Q}_{n-1}$ with degree higher than $(n-1)/2$ more subtle, where the problem of existence has been little understood up to now. 
The SVD method, however, enables us to construct plenty of examples and give a lower bound to the dimension of the moduli space for the higher degree case.
\begin{theorem}
For any $(n-1)/2< d\leq n-2$, the linearly full constantly curved holomorphic $2$-spheres of degree $d$ exist in $\mathcal{Q}_{n-1}$. Moreover, if $3\leq(n-1)/2<d\leq n-5$, then the moduli space of such noncongruent holomorphic $2$-spheres assumes $(n-d)^2-11(n-d)+33$  as the lower bound to its dimension.
\end{theorem}
For more precise description, see Theorem~\ref{thm-Hd}.

Our paper is organized as follows. Section~\ref{sec-pre} is devoted to reviewing known results on the SVD of complex matrices, the representation theory of $SU(2)$ with emphasis on the Clebsch-Gordan formula, and some basic formulas of minimal surfaces in the hyperquadric $\mathcal{Q}_{n-1}$. In Section~\ref{sec-orbit}, the orbit space of Grassmannian $G(d+1,n+1,\mathbb{C})$ with respect to the action of real orthogonal group $O(n+1;\mathbb{R})$ is determined by the SVD method. For noncongruent minimal 2-spheres of the same constant curvature, we investigate the structure of their moduli space in Section~\ref{sec-mod}, where how to construct minimal 2-spheres by the SVD method is introduced. Then Section~\ref{sec-const} is devoted to constructing constantly curved holomorphic 2-spheres of higher degree in $\mathcal{Q}_{n-1}$. For degree no more than 3, a complete classification is obtained in Section~\ref{sec-classify}. After studying more geometric properties of minimal 2-spheres constructed by employing the SVD method, {\bf Problem 1} and {\bf Problem 2} are discussed in Section~\ref{sec-prob}.

\section{Preliminaries}\label{sec-pre}
\subsection{Singular-value decomposition and unitary congruence.}
Let $M$ be an $n\times m$ complex matrix with $\corank(M)=r_0$. Set $q:=\min\{m,n\}$. It is well known that the eigenvalues of the Hermitian matrix $MM^{\ast}$, where $M^{\ast}$ is the conjugate transpose of $M$, are nonnegative real numbers. We denote them in nondecreasing order
\begin{equation*}
\lambda_{1}=\cdots=\lambda_{r_0}=0<\lambda_{r_0+1}\leq \lambda_{r_0+2}\leq\cdots\leq\lambda_{q}.
\end{equation*}
Set $\sigma_{i}=\sqrt{\lambda_{i}}.$
\begin{equation}\label{vec sigma}\sigma_1,\sigma_2,\cdots,\sigma_q, \quad or,\quad\vec{\sigma}=(\sigma_{1},\sigma_{2},\ldots,\sigma_{q}),
\end{equation}
are called, respectively, the \emph{singular values}, or, {\em singular-value vector}, of $M$. 
Using these notations, the singular-value decomposition (SVD) of $M$ can be stated as follows.
\begin{theorem}{\rm\cite[Thm.~2.1,~p.~150]{HornMatrAnalysis}}\label{singular value decomp lemma}
Let $M$ be an $n\times m$ complex matrix. Set $q=\min\{m,n\}$. Assume $\vec{\sigma}$ is the singular-value vector given in \eqref{vec sigma}. Let $\Sigma_{q}:=\diag(\sigma_{1},\ldots,\sigma_{q})$. Then there are unitary matrices $V\in U(n)$ and $W\in U(m)$, such that
\begin{equation}\label{svd equation}
M=V\,\Sigma \,W^{\ast},
\end{equation}
where $$
\Sigma=\begin{cases}
\begin{pmatrix}\Sigma_{q} & 0_{n\times(m-n)} \end{pmatrix},~~~&n<m,\\
~\Sigma_q,~~~&n=m,\\
^{t}\!\!\begin{pmatrix}\Sigma_{q} & 0_{m\times(n-m)} \end{pmatrix},~~~&n>m.
\end{cases}
$$

\end{theorem}

For us, the SVD of real and complex symmetric matrices are useful in the following.
\begin{coro}{\rm\cite[Cor.~2.6.7,~p.~154]{HornMatrAnalysis}\label{real svd theorem}}
Let $M$ be an $n\times m$ \emph{real} matrix. Under the same assumptions and notations as in Theorem {\rm\ref{singular value decomp lemma}}, there are real orthogonal matrices $V\in O(n;\mathbb{R})$ and $W\in O(m;\mathbb{R})$ satisfying \eqref{svd equation}.
\end{coro}

Two complex matrices $A$ and $B$ are said to be \emph{unitarily congruent} to each other \cite[p.~41]{HornMatrAnalysis} if there is a unitary matrix $U\in U(n)$ such that $$A=^{t}\!\!UB\,U, $$where $^{t}U$ is the transpose of $U$. It is clear that the singular values of $A$ and $B$ are identical. Thus, the singular values are invariant under unitary congruence. Moreover, for complex symmetric matrices, the singular values are the complete invariants.


\begin{theorem}\rm{\cite[p.~153,~Cor 2.6.6, p.~263, Cor 4.4.4]{HornMatrAnalysis}\label{singular values as unique invariant}}{\em
Let $A$ be an $n\times n$ complex symmetric matrix with $\corank(A)=r_0$. We denote the distinct positive singular values of $A$ by $\sigma_{1},\sigma_{2},\ldots,\sigma_{m}$, in increasing order, with multiplicities $r_1, r_2, \cdots, r_m,$ respectively.

{\bf (1)} There is a unitary matrix $U\in U(n)$ such that
\begin{equation}\label{dilation}
A=^{t}\!\!U\, 
    \diag(0_{r_0\times r_0},\,\sigma_{1}Id_{r_1},\,\cdots,\,\sigma_{m}Id_{r_m})\,
  U,
\end{equation}
and, moreover, if $\widetilde U$ is another such kind of matrix, then
$$\widetilde U=\diag(A_{r_0},\,A_{r_1},\,\cdots,\,A_{r_m})\,U,$$
where $A_{r_0}\in U(r_0)$ is unitary and $A_{r_j}\in O(r_j;\mathbb{R})$ is real orthogonal for any $1\leq j\leq m$.

{\bf (2)} Furthermore, two complex symmetric matrices are unitarily congruent to each other if and only if their singular values are the same.
}
\end{theorem}

\subsection{Irreducible representations of $SU(2)$.}\label{2.2}
Under the standard metric of constant curvature, the group of automorphisms of isometry of $\mathbb{C}P^{1}$ is $SU(2)$.
The irreducible representations of $SU(2)$ are well known \cite[Chap.~6]{GrpAndSymmetry} which we state briefly as follows.

Let $\mathcal{V}^{\frac{d}{2}}$ be the linear space of homogeneous polynomials of degree $d$ in two variables $(u,v)$, where $d$ is a nonnegative integer; we adopt the $d/2$ notation in accord with spin-$1/2$ in physics. The complex dimension of $\mathcal{V}^{\frac{d}{2}}$ is $d+1$, and we choose the following basis
\begin{equation}\label{basis of repre of su2}
e_{l}=\tbinom{d}{l}^{\frac{1}{2}}u^{d-l}v^{l},~~~l=0,\ldots,d.
\end{equation}
We equip $\mathcal{V}^{\frac{d}{2}}$ with an inner product such that $\{e_{l},~l=0,\ldots,d\}$ is an orthonormal basis. Consider the representation of $SU(2)$ on $\mathcal{V}^{\frac{d}{2}}$ given by
\begin{equation*}
 \varrho^{\frac{d}{2}}:SU(2)\times \mathcal{V}^{\frac{d}{2}} \rightarrow \mathcal{V}^{\frac{d}{2}}, \quad\quad\quad\quad
(g=\begin{pmatrix}
     a & b \\
     -\bar{b} & \bar{a} \\
   \end{pmatrix},~f
)\mapsto f\circ g^{\ast},
\end{equation*}
where $|a|^{2}+|b|^{2}=1$. Explicitly, $(\varrho^{\frac{d}{2}}(g)f)(u,v)=f(\bar{a}u-bv,\bar{b}u+av)$. Under the basis $\{e_{l},~l=0,\ldots,d\}$, by a slight abuse of notation, we set
\begin{equation}\label{irre repre in matrix form}
(e_{0},\ldots,e_{d})\cdot \rho^{\frac{d}{2}}(g):=\varrho^{\frac{d}{2}}(g)(e_{0},\ldots,e_{d}),
\end{equation}
 to indicate that $\rho^{\frac{d}{2}}(g)$ on the left hand side is a matrix in $U(d+1)$ (dot denotes matrix multiplication). Recall the rational normal (or, Veronese) curve $\mathbb{C}P^{1}$ of degree $d$,
\begin{equation}\label{rational normal curve}
 Z_{d}: \mathbb{C}P^{1}\mapsto \mathbb{C}P^{d},\quad\quad\quad\quad
 [u,v]\mapsto ~^{t}[u^{d},\ldots,\tbinom{d}{i}^{\frac{1}{2}}u^{d-i}v^{i},\ldots,v^{d}];
\end{equation}
Using the basis $e_{l}$ in \eqref{basis of repre of su2}, we can rewrite $Z_{d}=~^{t}[e_{0},\ldots,e_{d}].$ Consider a unitary transformation $U\in U(d+1)$ fixing the rational normal curve $Z_{d}$ with
\begin{equation}\label{image equivalence}
\text{Image}~U\cdot Z_{d}=\text{Image}~Z_{d}.
\end{equation}
(Henceforth, we use $Im$ to denote ``Image''). Since $U$ induces an isometric biholomorphic map of $\mathbb{C}P^{1}$, there are $A\in SU(2)$ and $\lambda\in U(1)$, such that
$U=\lambda\cdot ~^{t}\rho^{\frac{d}{2}}(A)$. The coefficient $\lambda$ here is necessary, because we consider projective transformations. The converse is also true. The geometric meaning of $U$ satisfying \eqref{image equivalence} is to reparametrize $Z_{d}$ by $A$.


The rational normal curve $Z_{d}$ is a special case of minimal 2-spheres called \emph{Veronese maps} in $\mathbb{C}P^{d}$ \cite{Bolton1988}. Explicitly,
\begin{align}\label{Veronese sequence}
\begin{split}
 &Z_{d,p}:\mathbb{C}P^{1}\rightarrow \mathbb{C}P^{d},\quad\quad\quad
[u,v]\mapsto [g_{p,0}(\frac{v}{u}),\cdots,g_{p,d}(\frac{v}{u})],\\
 &g_{p,l}(z)=\frac{p!}{(1+|z|^{2})^{p}}\sqrt{\tbinom{d}{l}}~z^{l-p}\sum_{k}(-1)^{k}\tbinom{l}{p-k}\tbinom{d-l}{k}|z|^{2k},\quad 0\leq p,l\leq d.
\end{split}
\end{align}
Note that $Z_{d,0}$ is the standard rational normal curve $Z_{d}$; we continue to denote it by $Z_{d}$ henceforth whenever convenient. We list some basic facts about the Veronese maps \cite[Section 2,~5]{Bolton1988} for easy reference. On the affine chart $u\neq 0$, set $z:=\frac{v}{u}$. Then $Z_{d,0}$ is given by
\begin{equation}\label{rational normal curve in affine chart}
Z_{d,0}=~^{t}[1,\sqrt{\tbinom{d}{1}}~z,\ldots,\sqrt{\tbinom{d}{k}}~z^{k},\ldots,z^{d}].
\end{equation}
$Z_{d,p+1}$ satisfies the following recursive formulas
\begin{equation}\label{recursive formulas}
Z_{d,p+1}=\frac{\partial}{\partial z}Z_{d,p}-\frac{\partial\log|Z_{d,p}|^{2}}{\partial z}Z_{d,p},~~~0\leq p\leq d-1.
\end{equation}
It follows that
\begin{equation}\label{local sections equivalent}
Z_{d,p}\equiv \frac{\partial^{p}}{\partial z^{p}}Z_{d,0} \mod (Z_{d,0},\ldots, Z_{d,p-1}),~~~
\end{equation}
for $0\leq p\leq d$. The norm squared of $Z_{d,p}$ is $|Z_{d,p}|^{2}=\frac{d!~p!}{(d-p)!}(1+|z|^{2})^{d-2p}$. Moreover,
\begin{equation}\label{facts of veronese sequence}
\frac{\partial}{\partial \bar{z}}Z_{d,p} =-\frac{|Z_{d,p}|^{2}}{|Z_{d,p-1}|^{2}}Z_{d,p-1}=-\frac{p(d-p+1)}{(1+|z|^{2})^{2}}Z_{d,p-1}.
\end{equation}

Equip $\mathbb{C}P^{d}$ with the Fubini-Study metric 
 of holomorphic sectional curvature $4$. Then the pullback metric of the Veronese map $Z_{d,p}$ is
\begin{equation}
ds_{d,p}^{2}=\frac{d+2p(d-p)}{(1+|z|^{2})^{2}}dzd\bar{z},
\end{equation}
and its Gaussian curvature $K_{d,p}$ is the constant $4/(d+2p(d-p))$. The K\"ahler angle of $Z_{d,p}$ is also constant, and we denote it by $\theta_{d,p}$, where $\theta_{d,p}\in [0,\pi]$, which satisfies \begin{equation}\label{eq-kahler}
\cos\theta_{d,p}=(d-2p)/(2p(d-p)+d),~0\leq p\leq d.
\end{equation}
Conversely, the constant curvature assumption characterizes $Z_{d,p}$ in the following rigidity theorem.
\begin{theorem}\label{rigidity theorem}{\rm\cite{Bolton1988}}
Let $f: \mathbb{C}P^{1}\rightarrow \mathbb{C}P^{d}$ be a linearly full minimal 2-sphere of constant curvature. Then there exist two unitary matrices $U\in U(d+1)$ and $B\in SU(2)$ 
such that
\begin{equation*}
f([u,v])=U\cdot ~^{t}\rho^{\frac{d}{2}}(B)\cdot Z_{d,p}([u,v]),~~~[u,v]\in \mathbb{C}P^{1}.
\end{equation*}


\end{theorem}

\subsection{Clebsch-Gordan formula.}
Consider the tensor product representation
\begin{align*}
\begin{split}
 \varrho^{\frac{d}{2}}\otimes \varrho^{\frac{d}{2}}:SU(2)\times \mathcal{V}^{\frac{d}{2}}\otimes \mathcal{V}^{\frac{d}{2}}&\rightarrow \mathcal{V}^{\frac{d}{2}}\otimes \mathcal{V}^{\frac{d}{2}}, \\
 (g, e_{k}\otimes e_{l})&\mapsto (\varrho^{\frac{d}{2}}(g)e_{k}) \otimes (\varrho^{\frac{d}{2}}(g)e_{l}), ~~~0\leq k,l\leq d.
\end{split}
\end{align*}
Identify $e_{k}\otimes e_{l}$ with the square matrix $E_{kl}\in M_{(d+1)}(\mathbb{C})$, where the only nonvanishing entry of $E_{kl}$ is $1$ at the $(k,l)$ position,~$0\leq k,l\leq d$. This identification gives rise to an isomorphism of $M_{(d+1)}(\mathbb{C})$ with $\mathcal{V}^{\frac{d}{2}}\otimes \mathcal{V}^{\frac{d}{2}}$,
\begin{equation}\label{identification of matrix with tensor product}
\varphi: M_{(d+1)}(\mathbb{C})\rightarrow \mathcal{V}^{\frac{d}{2}}\otimes \mathcal{V}^{\frac{d}{2}},\quad\quad\quad\quad
 A\mapsto ~(^{t}Z_{d}\cdot A) \otimes Z_{d},
\end{equation}
under which 
we may write $\varrho^{\frac{d}{2}}\otimes \varrho^{\frac{d}{2}}$ as 
\begin{align}\label{action on M_d}
\begin{split}
\varrho^{\frac{d}{2}}\otimes \varrho^{\frac{d}{2}}: SU(2)\times M_{(d+1)}(\mathbb{C})&\rightarrow M_{(d+1)}(\mathbb{C}), \\
 (g, A)&\mapsto ^{t}\!\!\rho^{\frac{d}{2}}(g)\cdot A\cdot \rho^{\frac{d}{2}}(g).
\end{split}
\end{align}
It is well known that by the Clebsch-Gordan formula, $\mathcal{V}^{\frac{d}{2}}\otimes \mathcal{V}^{\frac{d}{2}}$ is decomposed into irreducible $SU(2)$-invariant subspaces (see \cite[2.4,~p.~90]{GrpAndSymmetry},  
\begin{equation*}
\mathcal{V}^{\frac{d}{2}}\otimes \mathcal{V}^{\frac{d}{2}}\cong \underbrace{\mathcal{V}^{d}\oplus \mathcal{V}^{d-1}\oplus\cdots\oplus \mathcal{V}^{0}}_{d+1},
\end{equation*}
and the projection to the first summand $\mathcal{V}^{d}$ is given by the product of polynomials
\begin{equation}
 \Pi:\mathcal{V}^{\frac{d}{2}}\otimes \mathcal{V}^{\frac{d}{2}}\mapsto \mathcal{V}^{d},\quad\quad\quad
 f_{1}\otimes f_{2}\mapsto f_{1}\cdot f_{2}.
\end{equation}
The set of symmetric matrices $Sym_{d+1}(\mathbb{C})$ turns out to be an $SU(2)$-invariant subspace of $M_{d+1}(\mathbb{C})$ as follows.

\begin{lemma}\label{symmetric matrix space}
Let $d$ be a nonnegative integer. Assume that  $d=2[\frac{d}{2}]+r$, where $0\leq r\leq 1$ and $[\frac{d}{2}]$ is the greatest integer less than or equal to $\frac{d}{2}$. Then
\begin{equation*}
Sym_{d+1}(\mathbb{C})\cong \mathcal{V}^{d}\oplus \mathcal{V}^{d-2}\oplus \cdots \oplus \mathcal{V}^{r}.  
\end{equation*}
\end{lemma}
\begin{proof} (sketch)
Consider the induced action of $\mathfrak{su}(2)$ on $Sym_{d+1}(\mathbb{C})$ and extend it to $\mathfrak{sl}(2;\mathbb{C})$. 
Since 
the symmetric matrix $\frac{1}{2}(E_{kl}+E_{lk})$ corresponds to
\begin{equation*}
\frac{1}{2}(e_{k}\otimes e_{l}+e_{l}\otimes e_{k}),~0\leq k\leq l\leq d,
\end{equation*}
under $\varphi$, 
a multiplicity count of the eigenvalues of $J_{3}:=\diag(-1,1)/2$
\, in $\mathfrak{sl}(2;\mathbb{C})$
finishes off the proof.
\end{proof}

Denote the projection of $Sym_{d+1}(\mathbb{C})$ into $\mathcal{V}^{d-2k}$ by $\Pi_{k}$,
\begin{equation}\label{general projection}
\Pi_{k}:Sym_{d+1}(\mathbb{C})\rightarrow \mathcal{V}^{d-2k},
\end{equation}
where $0\leq k\leq [\frac{d}{2}]$. Recall the definition of the identification $\varphi$, \eqref{identification of matrix with tensor product}, where the projection $\Pi_{0}$ to the first summand $\mathcal{V}^{d}$ is given by the product of polynomials, or, in matrix terms, 
\begin{align}\label{projection to the first summand}
 \Pi_{0}:Sym_{d+1}(\mathbb{C})\mapsto \mathcal{V}^{d},\quad\quad\quad
 S\mapsto ~^{t}Z_{d}\cdot S\cdot Z_{d}.
\end{align}

Fix $0\leq p\leq [\frac{d}{2}]$. Consider the following subspace of $Sym_{d+1}(\mathbb{C})$,
\begin{equation}\label{important subsapce of symm matrices}
\mathscr{S}_{d,p}:=\{S\in Sym_{d+1}(\mathbb{C})| ~^{t}Z_{d,p}\cdot S\cdot Z_{d,p}=0\},
\end{equation}
which can be seen as the set of all quadrics containing the standard Veronese map $Z_{d,p}$. 
It turns out that it is $SU(2)$-invariant as follows. 
\begin{prop}\label{quadric contains rnc and decomposition}
Let $d$ be a nonnegative integer and let $0\leq p\leq [\frac{d}{2}]$. Then

 {\bf (1)} $\mathscr{S}_{d,p}=\ker \Pi_{0}\cap \dots\cap\ker \Pi_{p}$.

 {\bf (2)} Hence, $\mathscr{S}_{d,p}\cong \mathcal{V}^{d-2p-2}\oplus \mathcal{V}^{d-2p-4}\oplus \cdots \oplus \mathcal{V}^{d-2[\frac{d}{2}]}$, 

~~~~{ \rm (}{when $p=[\frac{d}{2}]$, the right hand side is understood to be $\{0\}$}{\rm)}.
\end{prop}
\begin{proof} For $p=0$, see \eqref{projection to the first summand}. The general case will be proven in Proposition \ref{summary theorem}.
\end{proof}

When $p=0$, for convenience, we denote $\mathscr{S}_{d,0}$ by $\mathscr{S}_{d}$ in this paper.

\subsection{Minimal surfaces in the hyperquadric.}\label{2.4}
Following the setup in Jiao and Wang \cite{JiaoWangQn}, let $f:M^{2}\rightarrow \mathcal{Q}_{n-1}\subseteq \mathbb{C}P^{n}$ be an isometric immersion. At every point $p\in M^{2}$, there is a local complex coordinate $z$ such that near that point
\begin{equation}\label{metric condition of surface in quadric}
ds^{2}_{M^{2}}=\lambda^{2}dzd\bar{z}.
\end{equation}
Assume $\widetilde{f}$ is a local lift of $f$ in $\mathbb{C}^{n+1}$ around $p$ with unit norm, $|\widetilde{f}|=1$. Consider the following two orthogonal projections of $\frac{1}{\lambda}\frac{\partial \widetilde{f}}{\partial z}$ and $\frac{1}{\lambda}\frac{\partial \widetilde{f}}{\partial \bar{z}}$ onto the subspace perpendicular to $\widetilde{f}$, respectively,
\begin{equation*}
X:=\frac{1}{\lambda}\frac{\partial \widetilde{f}}{\partial z}-\langle \frac{1}{\lambda}\frac{\partial \widetilde{f}}{\partial z},\widetilde{f} \rangle\widetilde{f},\quad\quad Y:=\frac{1}{\lambda}\frac{\partial \widetilde{f}}{\partial \bar{z}}-\langle \frac{1}{\lambda}\frac{\partial \widetilde{f}}{\partial \bar{z}},\widetilde{f} \rangle\widetilde{f},
\end{equation*}
where $\langle~~,~~\rangle$ denote the standard unitary inner product in $\mathbb{C}^{n+1}$. 
From the metric condition \eqref{metric condition of surface in quadric}, we know
\begin{align*}
|X|^{2}+|Y|^{2} &=1,\quad\quad\quad
 ^{t}\overline{X}~\cdot Y=0,
\end{align*}
where we identify $X$ and $Y$ as the column vectors in $\mathbb{C}^{n+1}$. The K\"ahler angle $\theta\in [0,\pi]$ of $f$ can be expressed as
\begin{equation}\label{kahler angle}
\cos \theta=|X|^2-|Y|^2.
\end{equation}
In \cite{JiaoWangQn}, the following global invariants are defined,
\begin{equation*}
\tau_{X}=|^{t}X\, X|,~~\tau_{Y}=|^{t}Y\, Y|,~~\tau_{XY}=|^{t}X\, Y|,
\end{equation*}
If $f$ is minimal in $\mathcal{Q}_{n-1}$, then from (2.22) in \cite[p.~821]{JiaoWangQn}, we obtain
\begin{equation}\label{formula of norm of second fundamental form}
||B||^{2}=2+6\cos^{2}\theta-2K-4\tau_{X}^{2}-4\tau_{Y}^{2}+8\tau_{XY}^{2},
\end{equation}
where $B$ is the second fundamental form of $f$.

As mentioned in the introduction, this paper is contributed to studying constantly curved 2-spheres minimal in both $\mathcal{Q}_{n-1}$ and $\mathbb{C}P^n$. The following characterization of Jiao, Wang and Zhong will be used in Section~\ref{sec-prob}.
\begin{theorem}\rm{\cite[Theorem 3.1]{zbMATH06272104}}\label{minimal in cpn and qn-1}
{\em The surface $f:M^{2}\rightarrow \mathcal{Q}_{n-1}\subseteq \mathbb{C}P^{n}$ is minimal in both $\mathcal{Q}_{n-1}$ and $\mathbb{C}P^n$, if and only if, $\tau_{XY}=0$.}


\end{theorem}
\begin{definition}
For two positive integers $d,n$ satisfying $d\leq n$, consider the set of minimal $2$-spheres of constant curvature, each of which spans a projective $d$-plane in $\mathbb{C}P^{n}$ and is minimal in both $\mathcal{Q}_{n-1}$ and $\mathbb{C}P^{n}$. For convenience, we denote it by $\mathbf{Mini}_{d,n}$.
\end{definition}

From the rigidity Theorem \ref{rigidity theorem} and Theorem \ref{minimal in cpn and qn-1}, we know such minimal 2-spheres are derived from the standard Veronese maps $Z_{d,p}$ with quadric constraint, for some $~0\leq p\leq d$, up to unitary transformations of $\mathbb{C}P^n$. This implies $\mathbf{Mini}_{d,n}$ has the structure of the following disjoint union.
\begin{prop}\label{set of mini dn}
\begin{equation*}
\mathbf{Mini}_{d,n}=\bigsqcup_{0\leq p\leq d}\mathbf{H}_{d,n,p},
\end{equation*}
where $\mathbf{H}_{d,n,p}$ is defined by
\begin{equation}\label{rnc in hyperquadric}
\mathbf{H}_{d,n,p}:=\{EZ_{d,p}|E\in M(n+1,d+1), ~E^{*}E=Id_{d+1},~\text{and }~EZ_{d,p}~ \text{lies in }~\mathcal{Q}_{n-1}\},
\end{equation}
where $E$ is of size $(n+1)\times (d+1)$. Moreover, by taking conjugation in $\mathbb{C}^{n+1}$, we have $\mathbf{H}_{d,n,p}\cong \mathbf{H}_{d,n,d-p},~0\leq p\leq [\frac{d}{2}]$.
\end{prop}
\begin{remark}
All minimal $2$-spheres belonging to $\mathbf{H}_{d,n,p}$ have the same constant curvature $K=4/(d+2p(d-p))$.
\end{remark}

In view of the second statement in the preceding proposition, it suffices to determine $\mathbf{H}_{d,n,p}$ for $0\leq p\leq [\frac{d}{2}]$, to be done in the following sections. 
It is geometrically clear that to study noncongruent minimal 2-spheres, we should mod out the real orthogonal group $O(n+1;\mathbb{R})$ and the reparametrizations of $\mathbb{C}P^1$.
To begin with, observe that a 2-sphere in ${\mathbf H}_{d,n,p}$ lies in the intersection, which is a quadric not necessarily smooth, of the hyperquadric $\mathcal{Q}_{n-1}$ and the projective $d$-plane it spans. 
Following this observation, we will first study the action of $O(n+1; \mathbb{R})$ on all projective $d$-planes, i.e., on $G(d+1,n+1,\mathbb{C})$.

\section{The orbit space $G(d+1,n+1,\mathbb{C})/O(n+1;\mathbb{R})$}\label{sec-orbit}
Throughout this section, $d$ is a positive integer. 
For a given plane $V\in G(d+1,n+1,\mathbb{C})$, we can choose an orthonormal basis $e_{0},\ldots,e_{d}$ of $V$ and represent it by an $(n+1)\times (d+1)$ matrix
\begin{equation}\label{representative of plane}
E:=\begin{pmatrix}
  e_{0} & \ldots & e_{d} \\
\end{pmatrix}.
\end{equation}
Two orthonormal bases of $V$ differ by a unitary matrix in $U(d+1)$ multiplied on the right of \eqref{representative of plane}, which implies that the complex symmetric matrices $^t\!EE$ differ by unitary congruence. This means that we can endow $V$ with $d+1$ numbers
$$\sigma_0\leq\sigma_1\leq\cdots\leq\sigma_d,$$
which are the singular values of\, $^t\!EE$ defining $n$ $O(n+1;\mathbb{R})$-invariant functions on $G(d+1,n+1,\mathbb{C})$. For convenience, in this paper, we also call $\{\sigma_0,\cdots,\sigma_{d}\}$ 
the \emph{singular values of $V$}, or the \emph{singular values of $O(n+1;\mathbb{C})$-orbit through $V$ in $G(d+1,n+1,\mathbb{C})$}, interchangeably. These invariants turn out to be decisive in constructing the $(d+1)$-plane $V$. Our main observation is inspired by the work of Berndt \cite[Section 6,~p. 27]{BerndtTwoGrassman}, as follows.

Let $M$ be a complete Riemannian manifold, and 
let $p\in M$ be a fixed point. Let $G$ be a compact Lie group with an action of isometry on $M$ given by,
$ G\times M\rightarrow M,\;
 (g,q)\mapsto g\cdot q.$
Set $N=G\cdot p$, the $G$-orbit thorough $p$ in $M$. It is a homogenous space $G/H$, where $H$ is the isotropy group of $G$ at $p$. 
We denote the normal space to $N$ at $p$ in $M$ by $T_p^\perp N$.
Note that $H$ also induces an isometric isotropy action on $T_{p}^{\perp}N$,
\begin{equation*}
 H\times T_{p}^{\perp}N \rightarrow T_{p}^{\perp}N,\quad\quad\quad\quad
 (g,v)\mapsto g_{\ast}|_{p}~v.
\end{equation*}

\begin{theorem}\label{general compute the orbit sapce}
For every point $q\in M$, there are $A\in G$ and $v\in T_{p}^{\perp}N$, such that
\begin{equation*}
q=A\cdot \exp_{p}v,
\end{equation*}
where $exp_{p}$ is the exponential map of the Riemannian manifold $M$ at $p$.

 Moreover, under the isotropy action, if $u$ lies in the same $H$-orbit of $v$, i.e., if there are $g\in H$ and $u\in T_{p}^{\perp}N$ such that $v=g_{\ast}|_{p}\cdot u$, then $q=Ag\cdot \exp_{p}u$.
\end{theorem}
\begin{proof}
Since $G$ is compact, the $G$-orbit $N$ through $p$ is also compact. Since $M$ is complete, for every point $q\in M$, there is a point $r\in N$, such that
$dist(q,r)=dist(q,N).$

Let $l(t)$ be a minimal geodesic connecting $r$ and $q$. By the first variation formula, $l(t)$ is a normal geodesic starting from $r$. So there is a normal vector $w\in T_{r}^{\perp}N$, such that $q=exp_{r}w$. Since $G$ acts on the $G$-orbit $N$ transitively, there is an $A\in G$ such that
$A\cdot p=r.$
Denote $A^{-1}_{\ast}|_{r}~w\in T_{p}^{\perp}N$ by $v$. Then we have $q=A\cdot \exp_{p}v$ since the action of $G$ is isometric. The second claim is now clear with slight modification.
\end{proof}


Using Theorem~\ref{general compute the orbit sapce}, we obtain the following decomposition result of unitary matrices.
\begin{prop}\label{norm form of unitary group}
Let $U$ be a unitary matrix in $U(n+1)$ with $n\geq d$.

 {\bf (1)} If $2d+1\leq n$, then there exist $A\in O(n+1;\mathbb{R}), U_{1}\in U(d+1),~U_{2}\in U(n-d)$, such that
\begin{equation}\label{general uintary low d}
U=A\begin{pmatrix}
     \Lambda_{1} & \Lambda_{2} & 0 \\
     \Lambda_{2} & \Lambda_{1} & 0 \\
     0 & 0 & Id_{n-2d-1} \\
   \end{pmatrix}\begin{pmatrix}
                                                                       U_{1} & 0 \\
                                                                       0 & U_{2} \\
                                                                     \end{pmatrix},
\end{equation}
where $\Lambda_{1}=\diag(\cos a_{0},\ldots,\cos a_{d}),~\Lambda_{2}=\diag(\sqrt{-1}\sin a_{0},\ldots,\sqrt{-1}\sin a_{d})$, for some $a_{j}\in [0,2\pi],~0\leq j\leq d$.

 {\bf (2)} If $2d+1>n$, set $l=2d+1-n$. Then there exist $A\in O(n+1;\mathbb{R}),~U_{1}\in U(d+1),~U_{2}\in U(n-d)$, such that
\begin{equation}\label{general uintary high d}
U=A\begin{pmatrix}
     \Lambda_{1} & 0 & \Lambda_{2} \\
     \Lambda_{2} & 0 & \Lambda_{1} \\
     0 & Id_{l} & 0 \\

   \end{pmatrix}\begin{pmatrix}
                                                                       U_{1} & 0 \\
                                                                       0 & U_{2} \\
                                                                     \end{pmatrix},
\end{equation}
   where $\Lambda_{1}=\diag(\cos a_{0},\ldots,\cos a_{d-l}),~\Lambda_{2}=\diag(\sqrt{-1}\sin a_{0},\ldots,\sqrt{-1}\sin a_{d-l})$, for some $a_{j}\in [0,2\pi],~0\leq j\leq d-l$.
\end{prop}
\begin{proof}
We consider only the case (1) with $2d+1\leq n$, the proof for the other case is similar. Equip $U(d+1)$ with the standard bi-invariant metric given by $(u,v)=Re~tr(u^{\ast}~v)$ over the Lie algebra $\mathfrak{u}(n+1)$.
Consider the group $G:=O(n+1;\mathbb{R})\times U(d+1)\times U(n-d)$ and its action on $U(n+1)$ given by
\begin{equation*}
 G\times U(n+1)\rightarrow U(n+1),\quad\quad\quad
 ((A,K,L)
 ,U)\mapsto AU\begin{pmatrix}
      K^{-1} &  \\
       & L^{-1} \\
    \end{pmatrix},
\end{equation*}
which preserves the metric of $U(n+1)$. 
We denote the $G$-orbit through the identity $Id\in U(n+1)$ by $N$. The isotropy group $H$ at $Id$ is isomorphic to $O(d+1;\mathbb{R})\times O(n-d;\mathbb{R})$,
\begin{equation*}
H=\{(\begin{pmatrix}
      A_{1} & 0 \\
      0 & A_{2} \\
    \end{pmatrix}
,A_{1},A_{2})|~A_{1}\in O(d+1;\mathbb{R}),~A_{2}\in O(n-d;\mathbb{R})\}.
\end{equation*}
It is easy to see that $T_{Id}N=\mathfrak{o}(n+1;\mathbb{R})+\mathfrak{u}(d+1)+\mathfrak{u}(n-d)$ (this is not a direct sum). Hence the normal space $T_{Id}^{\perp}N$ is given by
\begin{equation*}
T_{Id}^{\perp}N=\{\begin{pmatrix}
        0_{(d+1)\times(d+1)} & \sqrt{-1} ~^{t}B\\
        \sqrt{-1}~B & 0_{(n-d)\times(n-d)} \\
      \end{pmatrix}|~B\in M_{(n-d)\times(d+1)}(\mathbb{R})\}.
\end{equation*}
The induced isotropy action of $H$ on $T_{Id}^{\perp}N$ is $H\times T_{Id}^{\perp}N\rightarrow T_{Id}^{\perp}N$ given by
$$
 \begin{small}\Big((\begin{pmatrix}
      A_{1} & 0 \\
      0 & A_{2} \\
    \end{pmatrix}
,A_{1},A_{2}),\begin{pmatrix}
        0_{(d+1)\times(d+1)} & \sqrt{-1} ~^{t}B\\
        \sqrt{-1}~B & 0_{(n-d)\times(n-d)} \\
      \end{pmatrix}\Big)\mapsto \begin{pmatrix}
        0_{(d+1)\times(d+1)} & \sqrt{-1} ~A_{1}~^{t}BA_{2}^{-1}\\
        \sqrt{-1}~A_{2}BA_{1}^{-1} & 0_{(n-d)\times(n-d)} \\
      \end{pmatrix}.
\end{small}
$$
By Corollary \ref{real svd theorem} and the assumption $n-d\geq d+1$, without loss of generality, we may assume that $B$ is in the form
\begin{equation*}
B=\begin{pmatrix}
    \Lambda \\
    0_{(n-2d-1)\times(d+1)} \\
  \end{pmatrix},
\end{equation*}
where $\Lambda=\diag(a_{0},\ldots,a_{d})$, for some $a_{j}\in\mathbb{R},~0\leq i\leq d$. Then by a straightforward computation on matrix exponential while invoking Theorem \ref{general compute the orbit sapce}, we arrive at \eqref{general uintary low d}.
\end{proof}

With the above decomposition theorem of unitary matrices, we can now show how to reconstruct the $(d+1)$-plane $V\subset \mathbb{C}^{n+1}$ from singular values.

\begin{coro}\label{classification of plane}
Let $V$ be a fixed $(d+1)$-plane in $\mathbb{C}^{n+1}$. Denote its distinct singular values by $\sigma_0,\sigma_1,\cdots,\sigma_m$, in increasing order, with multiplicities $r_0, r_1, \cdots, r_m$, respectively.  Then all $\sigma_j \in [0,1]$, and there exist numbers $a_j\in [0, \frac{\pi}{4}]$ such that $\sigma_{j}=\cos2a_j$ for all $j$. Moreover,

{\bf (1)} if $2d+1\leq n$, then any orthonormal basis $E$ of $V$ can be expressed as $E=AV_{\vec{\sigma}}U$, where $A\in O(n+1;\mathbb{R})$, $U\in U(d+1)$, and
\begin{small}
\begin{equation}\label{jm alpha0}
V_{\vec{\sigma}}\doteq\begin{pmatrix}
\mathrm{J}_0(\sigma_0)~&&\\
&\ddots&\\
&&\mathrm{J}_m(\sigma_m)\\
0_{(n-2d-1)\times r_0}&\cdots&0_{(n-2d-1)\times r_m}\\
\end{pmatrix},~~~
\mathrm{J}_j(\sigma_j)\doteq\begin{pmatrix}
  \cos a_j \,Id_{r_j} \\
  \sqrt{-1}\sin a_j\,Id_{r_j} \\
\end{pmatrix},~0\leq j\leq m;
  \end{equation}
\end{small}

 {\bf (2)} if $2d+1>n$ and we set $l=2d+1-n$, then $\sigma_m=1$, $r_m\geq l$, and any orthonormal basis $E$ of $V$ can be expressed by $E=AV_{\vec{\sigma}}U$, where $A\in O(n+1;\mathbb{R})$, $U\in U(d+1)$, and
\begin{small}
\begin{equation}\label{jm alpha}
V_{\vec{\sigma}}\doteq\begin{pmatrix}
\mathrm{J}_0(\sigma_0)~&&&\\
&\ddots&&\\
&&\mathrm{J}_{m-1}(\sigma_{m-1})&\\
&&&Id_{r_m}\\
0_{(r_m-l)\times r_0}&\cdots&\cdots&0_{(r_m-l)\times r_m}
\end{pmatrix},
  \end{equation}
\end{small}
where $\mathrm{J}_j(\sigma_j)$ is defined similarly as in \eqref{jm alpha0}.
\end{coro}
\begin{proof}
~We only deal with the case of $2d+1\leq n$; the proof for the case of $2d+1>n$ is similar.
Assume $E=(e_{0},\ldots,e_{d})$ is an orthonormal basis of $V$. Then there is a $U\in U(n+1)$, such that
\begin{equation*}
E=U\begin{pmatrix}
          Id_{d+1} \\
          0 \\
        \end{pmatrix}.
\end{equation*}
By \eqref{general uintary low d}, we obtain \eqref{jm alpha0} for some real parameter $(a_{0},\ldots,a_{d})\in\mathbb{R}^{d+1}$, where $a_{j}\in[0,2\pi]$. We leave it to the reader as a routine exercise to show that 
by multiplying $O(n+1;\mathbb{R})$ on the left and $U(d+1)$ on the right of \eqref{jm alpha0}, we may assume $a_{i} \in[0,\frac{\pi}{4}]$.
\end{proof}
\begin{remark}\label{rk-nonfull}
From the above proposition, we see that if $2d+1<n$, then there exists a universal $\mathbb{C}^n\subset\mathbb{C}^{n+1}$ such that any $(d+1)$-planes in $\mathbb{C}^{n+1}$ can be transformed into $\mathbb{C}^n$ by orthogonal transformations. So, to consider the orbit space $G(d+1,n+1,\mathbb{C})/O(n+1;\mathbb{R})$, we may assume $2d+1\geq n$ in the following.
\end{remark}

Now the structure of $G(d+1,n+1,\mathbb{C})/O(n+1;\mathbb{R})$  follows from Corollary~\ref{classification of plane} directly. 

\begin{theorem}\label{moduli of plane}


If $d\leq n\leq 2d+1$, then the orbit space
$$G(d+1,n+1,\mathbb{C})/O(n+1;\mathbb{R})\cong \Theta_{d,n},$$
where
\begin{small}
\begin{equation}\label{eq-Theta}
\Theta_{d,n}:=\{\vec{\sigma}=(\sigma_{0},\sigma_{1},\cdots,\sigma_{d})\in \mathbb{R}^{d+1}|~0\leq \sigma_{0}\leq \sigma_{1}\leq\cdots\leq  \sigma_{d}\leq 1,~\sigma_{n-d}=\cdots=\sigma_{d}=1\},
\end{equation}
\end{small}
which is a convex polytope in $\mathbb{R}^{d+1}$.
\end{theorem}

We point out that the isotropic $2$-planes in $Q_{4}$ given by the linear spans of $(e_i\pm\sqrt{-1}e_i)/\sqrt{2}, 1\leq i\leq 3$, respective of the sign, relative to the standard basis for ${\mathbb C}^6$, shows that two planes corresponding to the same parameters in $\Theta_{d,n}$ need not be congruent under the action of $SO(n+1;\mathbb{R})$. This is the key reason why we chose the group $O(n+1;\mathbb{R})$. 
\begin{remark}
For a given $\vec{\sigma}\in \Theta_{d,n}$, the orbit determined by it can also be described in a clear way through computing the isotropic group of $V_{\vec{\sigma}}$ {\rm(}see \eqref{jm alpha}{\rm)}. Since it will not be used in the following, we omit it here, only to point out that the computation is similar to the proof of Lemma~{\rm\ref{lemma-equiv}} in the next section.
\end{remark}

\section{The moduli space of constantly curved 2-spheres minimal in both $\mathcal{Q}_{n-1}$ and $\mathbb{C}P^n$}\label{sec-mod}
Let $d,n$ be two positive integers with $d\leq n$, and let $p$ be a nonnegative integer with $0\leq p\leq [\frac{d}{2}]$.
Suppose $\gamma\in \mathbf{H}_{d,n,p}$. 
From~\eqref{rnc in hyperquadric}, we see $\gamma$ can be parameterized as $\gamma=EZ_{d,p}$, where $E$ is an orthonormal basis of the projective $d$-plane spanned by $\gamma$. It follows from Remark~\ref{rk-nonfull} that if $2d+1<n$, then $\gamma$ is not linearly full. 
To consider the linearly full minimal 2-spheres, we make the following convention in the subsequent part of this paper,
$$n\leq 2d+1,~~~l:=2d+1-n.$$
The structure of $\mathbf{H}_{d,n,p}$ is now clear following the expression of $E$ given in Corollary~\ref{classification of plane}.
\begin{prop}\label{classification of cdn}
\begin{equation}
\mathbf{H}_{d,n,p}=\{EZ_{d,p}|~E\in M(n+1, d+1),~E^{*}E=Id_{d+1},~~^{t}\!EE\in \mathbf{S}_{d,n,p}\},
\end{equation}
where $\mathbf{S}_{d,n,p}$ is a closed subset of $\mathscr{S}_{d,p}$ {\rm(}see~\eqref{important subsapce of symm matrices} for notation{\rm)}, defined by
\begin{equation*}
\{S\in \mathscr{S}_{d,p}| ~\text{$S$ takes 1 as its maximal singular value with multiplicity no less than } 2d+1-n\}.
\end{equation*}
\end{prop}
Here, we make another convention that will be used in the subsequent part of this paper,

$
~~~~~~~~~~~~~~~~~~~~~~~~~~~~~~~~~\mathbf{H}_{d,n}:=\mathbf{H}_{d,n,0},~~~~~~\mathbf{S}_{d,n}:=\mathbf{S}_{d,n,0}.
$
\begin{remark}\label{construction}
Conversely, for a given symmetric matrix $S\in \mathbf{S}_{d,n,p}$ with $\corank(S)=r_0$, suppose all the distinct positive singular values of $S$ are given by
$$\cos 2a_1< \cos 2a_2< \cdots< \cos 2a_{m},$$
with multiplicities $r_1, r_2, \cdots, r_m$, respectively, where $a_j\in[0,\frac{\pi}{4}]$.
Consider the SVD of $S$. It follows from Theorem~{\rm\ref{singular values as unique invariant} that there exists a unitary matrix $U\in U(d+1)$} such that
\begin{equation}\label{eq-SVD-S}
S=^{t}\!\!U\diag(0_{r_0\times r_0},\,\cos 2a_1I_{r_1},\,\cos 2a_2I_{r_2},\, \cdots , \,\cos 2a_{m}I_{r_m})\,U.
\end{equation}
It is easy to verify that
\begin{equation}\label{ex-construction}
V_{\vec{\sigma}}\,U\,Z_{d,p} \in \mathbf{H}_{d,n,p},
\end{equation}
where we have used the notation $V_{\vec{\sigma}}$ introduced in \eqref{jm alpha0} and \eqref{jm alpha}.
\end{remark}
\begin{lemma}\label{lemma-equiv}
If $\widetilde U\in U(d+1)$ is another unitary matrix in the SVD \eqref{eq-SVD-S} of $S$,
then there exists a matrix $A\in O(n+1;\mathbb{R})$ such that
$$AV_{\vec{\sigma}}U=V_{\vec{\sigma}}\widetilde U.$$
\end{lemma}
\begin{proof}
It follows from Theorem~\ref{singular values as unique invariant} that
$$\widetilde U=\diag(A_{r_0},A_{r_2},\cdots,A_{r_m})U,$$
where $A_{r_0}\in U(r_0)$ is unitary and $A_{r_j}\in O(r_j;\mathbb{R})$ is real orthogonal for any $1\leq j\leq m$.
Note that $V_{\vec{\sigma}}$ can also be written as
$$V_{\vec{\sigma}}=\diag(\mathrm{J}_0,\mathrm{J}_1,\cdots, \mathrm{J}_{m-1},\mathrm{J}_m),$$
where $\mathrm{J}_0=^t\!\!\begin{pmatrix}
  \frac{1}{\sqrt{2}} \,Id_{r_0},
  \frac{\sqrt{-1}}{\sqrt{2}}\,Id_{r_0}
\end{pmatrix}$, $\mathrm{J}_j=^t\!\!\begin{pmatrix}
 \cos a_j \,Id_{r_j},
  \sqrt{-1}\sin a_j\,Id_{r_j}
\end{pmatrix}$
for $1\leq j\leq m-1$, and
$\mathrm{J}_m=^t\!\!\begin{pmatrix}
 \cos a_m \,Id_{r_m},
  \sqrt{-1}\sin a_m\,Id_{r_m}
\end{pmatrix}$ or $\mathrm{J}_m=^t\!\!\begin{pmatrix}
  Id_{r_m} ,
  0_{ r_m\times(r_m-l)} \\
\end{pmatrix}$.

Now the conclusion follows from a straightforward verification that
$$\mathrm{J}_0A_{r_0}=\begin{pmatrix}
  \frac{1}{\sqrt{2}} \,Id_{r_0}\\
  \frac{\sqrt{-1}}{\sqrt{2}}\,Id_{r_0}
\end{pmatrix}A_{r_0}=\begin{pmatrix}
  Re(A_{r_0})&Im(A_{r_0})\\
  -Im(A_{r_0})&Re(A_{r_0})
\end{pmatrix}\begin{pmatrix}
  \frac{1}{\sqrt{2}} \,Id_{r_0}\\
  \frac{\sqrt{-1}}{\sqrt{2}}\,Id_{r_0}
\end{pmatrix},$$
$\mathrm{J}_jA_{r_j}=\diag(A_{r_j}, A_{r_j})\mathrm{J}_j$ for $1\leq j\leq m-1$, and $\mathrm{J}_mA_{r_m}=\diag(A_{r_m}, B)\mathrm{J}_m$ with $B=A_{r_m}$ or $B=Id_{r_m-l}$.
\end{proof}

To build a clearer relation between $\mathbf{H}_{d,n,p}$ and $\mathbf{S}_{d,n,p}$, some equivalences need to be introduced.

Two minimal 2-spheres in $\mathbf{H}_{d,n,p}$ are said to be equivalent if they are congruent in $\mathcal{Q}_{n-1}$, i.e., if one can be brought to the other by some $A\in O(n+1,\mathbb{R})$ and some $SU(2)$-reparametrization of ${\mathbb C}P^1$. For convenience, denote by $\mathbf{H}_{d,n,p}/O(n+1;\mathbb{R})$ the set of all equivalence classes. We point out that
\begin{equation*}
\mathbf{Mini}_{d,n}/O(n+1;\mathbb{R})=\bigsqcup_{0\leq p\leq d}\mathbf{H}_{d,n,p}/O(n+1;\mathbb{R}),
\end{equation*}
describes the moduli space of all noncongruent 2-spheres of constant curvature, minimal in both $\mathcal{Q}_{n-1}$ and $\mathbb{C}P^n$.

The equivalence on $\mathbf{S}_{d,n,p}$ is defined by the following group action, which is induced by \eqref{action on M_d},
\begin{align}\label{action on mathbf s_d}
 \varrho_{d}: U(1)\times SU(2)\times \mathbf{S}_{d,n,p}\rightarrow \mathbf{S}_{d,n,p},\quad\quad\quad
 (\lambda,g, S)\mapsto \lambda\, ^{t}\!\!\rho^{\frac{d}{2}}(g)\cdot S\cdot \rho^{\frac{d}{2}}(g).
\end{align}
The orbit space of $\varrho_{d}$ is denoted by $\mathbf{S}_{d,n,p}/U(1)\times SU(2)$.

From the above definition of equivalence, it is easy to see that the map
$$\mathbf{H}_{d,n,p}/O(n+1)\longrightarrow \mathbf{S}_{d,n,p}/U(1)\times SU(2),~~~~~~[\gamma]\mapsto [^{t}WW],$$
is well-defined with its inverse given by
$$\mathbf{S}_{d,n,p}/U(1)\times SU(2)\longrightarrow\mathbf{H}_{d,n,p}/O(n+1),~~~~~~[S]\mapsto [V_{\vec{\sigma}}UZ_{d,p}],$$
where $\vec{\sigma}$ is the singular-value vector of $S$, $U\in U(d+1)$ is a unitary matrix coming from the SVD \eqref{eq-SVD-S} of $S$, and the well-definedness of this inverse mapping follows from Lemma~\ref{lemma-equiv}.

\begin{theorem}\label{simple classification} The moduli space $\mathbf{H}_{d,n,p}/O(n+1;\mathbb{R})$ is given by
\[\mathbf{H}_{d,n,p}/O(n+1;\mathbb{R})\cong\mathbf{S}_{d,n,p}/U(1)\times SU(2).\]
\end{theorem}

A special case is $p=0$, for which $\mathbf{H}_{d,n}/O(n+1;\mathbb{R})$ is the moduli space of all noncongruent holomorphic 2-spheres of constant curvature and degree $d$ in $\mathcal{Q}_{n-1}$. As a consequence of Theorem~\ref{simple classification}, we reprove the main results of \cite{M-N-T}.
\begin{theorem}\label{coro-dim}
~

{\bf (1)} If $2d+1<n$, then the holomorphic $2$-spheres of constant curvature and degree $d$

~~~~~in $\mathcal{Q}_{n-1}$ is not linearly full.

{\bf (2)}\label{con-2} $\bigcup_{n<2d+1}\mathbf{H}_{d,n}/O(n+1;\mathbb{R})$ constitutes the boundary of\, $\mathbf{H}_{d,2d+1}/O(n+1;\mathbb{R})$, and
\begin{equation}\label{eq-dimcount}
\dim(\mathbf{H}_{d, 2d+1}/O(2d+2;\mathbb{R}))=d^2-d-4.
\end{equation}
\end{theorem}
\begin{proof}
The first conclusion follows from Remark~\ref{rk-nonfull} as pointed out in the beginning of this section. In view of Theorem~\ref{simple classification}, to prove the second conclusion, it suffices to analyze $\mathbf{S}_{d,n}$. By definition, the boundary of $\mathbf{S}_{d,2d+1}$ is comprised of all $\mathbf{S}_{d,n}$ with $n<2d+1$. The dimension count follows from

$~~~~~~~~~~~~~~~~\dim\mathbf{S}_{d,2d+1}=\dim\mathscr{S}_{d}=\dim Sym_{d+1}(\mathbb{C})-2(2d+1)=d^2-d.$
\end{proof}
\begin{remark}\label{rk-geo-exp}
We give a geometric explanation to the structure of the boundary of the moduli space $\mathbf{H}_{d,2d+1}/O(n+1;\mathbb{R})$. Since we can naturally embed a lower-dimensional complex projective space into higher dimensional ones, it is readily seen that
$$\mathbf{H}_{d,d}\subset\mathbf{H}_{d,d+1}\subset\cdots\subset\mathbf{H}_{d,2d}\subset\mathbf{H}_{d,2d+1}.$$
Conversely,
take a holomorphic $2$-sphere $\gamma\in\mathbf{H}_{d,2d+1}$ on the boundary. Then the maximal singular values of the $d+1$ plane $V$ spanned by $\gamma$ in $\mathbb{C}^{n+1}$ is $1$, whose multiplicity is denoted by $r_m$. It follows from \eqref{jm alpha0} in Corollary~{\rm\ref{classification of plane}} that there exists a real orthogonal transformation $A$ such that $A\gamma\in \mathbf{H}_{d,2d+1-r_m}.$
\end{remark}

Theorem~\ref{coro-dim} shows that there is an abundance of noncongruent holomorphic $2$-spheres of constant curvature and degree $d$ in $\mathcal{Q}_{2d}$. We give a more detailed account of this.
\begin{prop}
Suppose $d\geq 4$, and $V$ is a generic $d$-plane in $\mathbb{C}P^{2d+1}$ containing a holomorphic $2$-sphere $\gamma\in \mathbf{H}_{d,2d+1}$. Then in $\mathbf{H}_{d,2d+1}$, there is at least a $(d^{2}-2d-5)$-dimensional family of holomorphic $2$-spheres lying in $V$, which are all noncongruent to each other.
\end{prop}
\begin{proof}
From Section~\ref{sec-orbit}, we know that the $d$-plane in $\mathbb{C}P^{2d+1}$ can be determined by its singular values up to real orthogonal transformations. Combining this with Theorem~\ref{simple classification}, we need only prove this lower bound of dimension estimate for complex symmetric matrices with the same singular values in $\mathbf{S}_{d,2d+1}$ modulo the action of $U(1)\times SU(2)$.

Note that the singular values give a semialgebraic map from $\mathbf{S}_{d,2d+1}$ to $\Theta_{d,2d+1}$ (see \eqref{eq-Theta} for definition), which are both semialgebraic sets. The local triviality theorem for semialgebraic maps \cite{Hardt} implies that for generic point in the image of this map, the dimension of its preimage is bigger than or equal to

$~~~~~~~~~~~~~~~~~~~~~~~~~~~\dim\mathbf{S}_{d,2d+1}-\dim\Theta_{d,2d+1}=d^2-2d-1.$
\end{proof}


As mentioned in the introduction, the ideal of a rational normal curve (holomorphic 2-sphere) of degree $d$ is generated by $d^2-d$ independent quadrics. Note that a rational normal curve $\gamma\in \mathbf{H}_{d,n}$, if and only if,  the quadric given by the intersection of $\mathcal{Q}_{n-1}$ with the projective $d$-plane spanned by $\gamma$ belongs to the ideal of this curve. Conversely, from Proposition~\ref{classification of cdn}, we see that to guarantee that a quadric in this ideal lies in $\mathcal{Q}_{n-1}$, there are no other constraints if $n=2d+1$, whereas there are more constraints when $n<2d+1$ as the symmetric matrix determined by this quadric takes $1$ as its maximal singular values.
The SVD method introduced in Remark~\ref{construction} proves effective to handle this problem, which will be shown in the next section.

\section{Constantly curved holomorphic 2-spheres of higher degree in $\mathcal{Q}_{n-1}$}\label{sec-const}

In this section, we consider the existence of linearly full holomorphic 2-spheres of constant curvature and higher degree ($d>\frac{n-1}{2}$) in $\mathcal{Q}_{n-1}$. Similar to the discussion in Remark~\ref{rk-geo-exp}, it is easy to verify that such holomorphic 2-spheres belong to $\mathbf{H}_{d,n}\setminus\mathbf{H}_{d,n-1}$.

From the construction method given in the last section (see Remark~\ref{construction}), we see that to construct examples in $\mathbf{H}_{d,n}\setminus\mathbf{H}_{d,n-1}$, a matrix $S\in \mathbf{S}_{d,n}\setminus\mathbf{S}_{d,n-1}$ need to be found. To be precise, we need to construct a complex symmetric matrix $S:=(s_{i,\,j})$ which takes $1$ as its maximal singular value with multiplicity $2d+1-n>0$ and satisfying the following equations
\begin{equation}\label{requirement}
\sum_{i+j=k}s_{i,\,j}\;\sqrt{\tbinom{d}{i}\;\tbinom{d}{j}}=0,~~~0\leq k\leq 2d.
\end{equation}

Observe that the more repetition of the maximum singular value 1, the more severe restriction it imposes on the set  $\mathbf{H}_{d,n}$ when $d\leq n<2d+1$. In fact, the results of Li and Jin \cite{zbMATH05590797} shows that $\mathbf{H}_{4,5}$ is nonempty while $\mathbf{H}_{5,5}$ is; see also Proposition~\ref{Pr} for a singular-value proof of this fact. Nevertheless,
we can prove the following.

\begin{theorem}\label{thm-Hd}
If $d\geq 3$, then
$$\varnothing\neq \mathbf{H}_{d,d+2}\subsetneqq \mathbf{H}_{d,d+3}\subsetneqq\cdots\subsetneqq \mathbf{H}_{d,2d+1}.$$
Moreover, ignoring the action of the orthogonal group and reparametrizations of\, $\mathbb{C}P^1$, 
we have the following dimension estimates that, $\dim(\mathbf{H}_{d, d+4})\geq3$ and
\[\dim(\mathbf{H}_{d, n+1}\setminus {\mathbf H}_{d,n})\geq (n-d)^2-11(n-d)+33,~~~d+4\leq n\leq 2d.\]
\end{theorem}

\begin{remark}\label{d,d+2rmk}
That the chain starts with ${\mathbf H}_{d,d+2}$ is essential, since ${\mathbf H}_{d,d}$ and ${\mathbf H}_{d,d+1}$ may be empty when $d$ is odd, while ${\mathbf H}_{d,d+1}$ may be equal to ${\mathbf H}_{d,d}$ when $d$ is even; see Remark~\rm{\ref{lijin}}.
\end{remark}

To establish the theorem, the following lemmas are needed.
\begin{lemma}\label{elme-lemma1}
Suppose $0< \epsilon_0<\epsilon_1<\cdots<\epsilon_k, k\geq 1,$ are some constants. For any given numbers
$$0\leq x_0\leq x_1\leq \cdots \leq x_{k-1}\leq 1,$$
there exists a number $x_k, k\geq 1,$ such that $0\leq x_k <x_{k-1}$, and, moreover, it solves
\begin{equation}\label{equxk1}
\aligned
&x_k\epsilon_k =
\sum_{j=0}^{(k-2)/2}(x_{2j+1}\epsilon_{2j+1}-x_{2j}\epsilon_{2j}),\quad
\text{if}\;\; k\;\text{is even, or},\\
&x_k\epsilon_k =
\sum_{j=1}^{(k-1)/2}(x_{2j}\epsilon_{2j}-x_{2j-1}\epsilon_{2j-1})+x_0\epsilon_0,\quad
\text{if}\;\;k\; \text{is odd}.
\endaligned
\end{equation}
\end{lemma}

\begin{proof}
With our assumption, it is easy to see that $x_k$ solves~\eqref{equxk1} 
is nonnegative. On the other hand, we can transform \eqref{equxk1} by
$$x_k\epsilon_k =-x_0\epsilon_0-
\sum_{j=0}^{(k-2)/2}(x_{2j}\epsilon_{2j}-x_{2j-1}\epsilon_{2j-1})+x_{k-1}\epsilon_{k-1}\leq x_{k-1}\epsilon_{k-1},$$
which implies $x_k<x_{k-1}$. 
\end{proof}

\begin{remark}\label{lemma1-remark}
In the lemma, it is readily seen that if $x_0$ is chosen to be positive, then $x_k$ is also positive.
\end{remark}

\begin{lemma}\label{elem-lemma2}
Suppose  $0<\epsilon_0<\epsilon_1<\cdots<\epsilon_k, k\geq 1,$ are some constants and $\epsilon_k< 2\epsilon_{k-1}$. For any given numbers
$$0\leq x_0\leq x_1\leq \cdots \leq x_{k-2}\leq 1,$$
and $0\leq x_k \leq 1,$ there exists a number $x_{k-1}$ such that $-1< x_{k-1} <1$, and, moreover, it solves
\begin{equation}\label{equxk3}
\aligned
&2x_{k-1}\epsilon_{k-1} =
2\sum_{j=0}^{(k-3)/2}(x_{2j+1}\epsilon_{2j+1}-x_{2j}\epsilon_{2j})-x_k\epsilon_k,\quad
\text{if}\;\; k\; \text{is odd, or,}\\
&2x_{k-1}\epsilon_{k-1} =
2\sum_{j=1}^{(k-2)/2}(x_{2j}\epsilon_{2j}-x_{2j-1}\epsilon_{2j-1})+2x_0\epsilon_0-x_k\epsilon_k,\quad
\text{if}\;\; k\; \text{is even}.
\endaligned
\end{equation}
\end{lemma}

\begin{proof}
The proof is similar to that of Lemma~\ref{elme-lemma1} after transposing the $x_k\epsilon_k$ term from the right hand to the left hand in \eqref{equxk3}. 
\end{proof}

\begin{remark}\label{lemma2-remark}
In the above lemma, it is straightforward to verify that if $x_0, x_1, \cdots, x_{k-2}, x_k$ are chosen to satisfy
$$
\text{either}\;\;0<x_0\leq x_1\leq \cdots \leq x_{k-2}\leq\frac{x_k}{2}\leq \frac{1}{2},\quad \text{or} \;\;
x_k=0<x_0\leq x_1\leq\cdots\leq x_{k-2}\leq 1,
$$
then $x_{k-1}$ determined by them is nonzero.
\end{remark}

\begin{lemma}\label{matrix-lemma}
For any given integers $l\geq 0,r\geq 0$ and numbers
$$0\leq \cos 2a_1\leq \cdots\leq \cos 2a_{r}<1,$$
if $m:=2r+l+1\leq d$ and $d\geq 3$, then we can solve the following equation
\begin{equation}\label{eqsm}
\sum_{j=0}^{m}s_{j,m-j}\sqrt{\tbinom{d}{j}}\sqrt{\tbinom{d}{m-j}}=0,
\end{equation}
such that $s_{j,m-j}=s_{m-j,j}$ and
$|s_{j,m-j}|=|s_{m-j,j}|=\cos2a_{j+1},~~~0\leq j\leq r-1,$
while for other $r\leq j\leq m-r$, $|s_{j,m-j}|$ takes value in
$\{\underbrace{1, 1,\cdots,1}_l, \underbrace{\lambda, \lambda}_2\}.$
Here, $\lambda\in (-1,1)$ is determined by $\{\cos 2a_1, \cdots, \cos 2a_{r}\}$.
\end{lemma}
\begin{proof}
Set $k:=[\frac{m}{2}]$ and
$\epsilon_j:=\sqrt{\tbinom{d}{j}\tbinom{d}{m-j}},~~0\leq j\leq k.$
Using the combinatorial identity
$$\tbinom{d}{j}\tbinom{d}{m-j}\tbinom{2d}{d}=\tbinom{m}{j}\tbinom{2d-m}{d-j}\tbinom{2d}{m},$$
we obtain that
$0<\epsilon_0<\epsilon_1<\cdots<\epsilon_k.$
Moreover, it is easy to verify that if $d\geq 3$ and $m$ is an even number, then
$\epsilon_k<2\epsilon_{k-1}.$
Now the conclusion follows from Lemma~\ref{elme-lemma1} and Lemma~\ref{elem-lemma2}.
\end{proof}

We are ready to prove Theorem~\ref{thm-Hd}.

\vspace{1mm}

\begin{proof}
Set $l:=2d+1-n$, we need only prove that for any given $0\leq l\leq d-2$, there exists a symmetric complex matrix $S:=(s_{i,j})$ solving \eqref{requirement}
and taking $1$ as its maximal singular value with multiplicity $l$.

In Lemma~\ref{matrix-lemma}, take $r=[\frac{d-l-1}{2}]=[\frac{n-d}{2}]-1$. Then $m:=2r+l+1$ is equal to $d-1$ when $d-l$ is even, or to $d$ when $d-l$ is odd. For any given two groups of numbers
$$0\leq \cos 2a_1\leq \cdots\leq \cos 2a_{r}<1,\quad\quad\quad 0\leq \cos 2b_1\leq \cdots\leq \cos 2b_{r}<1,$$
we solve \eqref{eqsm} to get $\{s_{0, m},~s_{1, m-1}, \cdots,s_{m-1,1},s_{m,0}\}$ and
$\{t_{0, m},~t_{1, m-1}, \cdots,t_{m-1,1},t_{m,0}\},$ with which two solutions
of the required symmetric matrix can be given by
\begin{equation}\label{eq-S}
S=\diag\Big\{\begin{pmatrix}
A&0&^tC\\
0&D&0\\
C&0&B
\end{pmatrix},\,0_{(d-m)\times(d-m)}\Big\},
\end{equation}
where $A, B\in M(r,r)$ are two complex symmetric matrices, and $C:=(c_{i,j})\in M(r,r)$ is a complex matrix with prescribed anti-diagonal as
$$c_{j, r-1-j}=\mu \,s_{j, m-j}+\sqrt{\mu^2-1}\,t_{j,m-j},~~~0\leq j\leq r-1,$$
and $D$ is the matrix
\begin{equation*}
\aligned
&\mu\,\antidiag\Big\{s_{r, m-r},~s_{r+1, m-r-1}, \cdots,s_{m-r,r}\Big\}\\
&+\sqrt{\mu^2-1}\,\antidiag\Big\{t_{r, m-r},~t_{r+1, m-r-1}, \cdots,t_{m-r,r}\Big\};
\endaligned
\end{equation*}
here $\mu\in[0,1]$ is a  parameter. It is easy to see that there are many $\{A, B, C\}$ to be chosen to satisfy \eqref{requirement}.

The singular values of $S$ defined in \eqref{eq-S} are
$$\{\underbrace{1, 1,\cdots,1}_l,\, \underbrace{\lambda, \lambda}_2,\,\underbrace{0}_{d-m},\,\sigma_1, \sigma_2,\cdots, \sigma_{2r}\},$$
where $\lambda\in (-1,1)$ is determined by $\{\cos 2a_1, \cdots, \cos 2a_{r},\cos 2b_1,\cdots,\cos 2b_{r}\}$, and
$0\leq \sigma_1\leq \sigma_2\leq\cdots\leq \sigma_{2r}$ are singular values of
$\begin{pmatrix}
A&^tC\\
C&B
\end{pmatrix}.$
Note that as a matrix norm, $\sigma_{2r}$ is no more than the Frobenius norm, which can be controlled to be less than $1$ (this is an open condition) with suitable choice of $\{A, B, C\}$. Hence the matrix $S$ we have constructed can take $1$ as its maximal singular value with multiplicity $l$.

Take such a matrix $S$ and let
$\lambda:=\cos 2\theta_{0},~~\sigma_{j}:=\cos 2\theta_j,~~1\leq j\leq 2r.$ Then we obtain a linearly full constantly curved holomorphic 2-sphere of degree $d$ in $\mathcal{Q}_{n-1}$ given by
$$
\gamma_d=
\begin{pmatrix}
    \Lambda_{1}&& \\
    \Lambda_{2}&& \\
    &&Id_l
\end{pmatrix}UZ_d,
  ~~~~\text{or}~~~~
\begin{pmatrix}
    \Lambda_{1}&&& \\
    \Lambda_{2}&&& \\
    &&\frac{1}{\sqrt{2}}&\\
    &&\frac{1}{\sqrt{-2}}&\\
    &&&Id_l
\end{pmatrix}UZ_d,
$$
where $\Lambda_{1}=\diag\{\cos \theta_{0},\ldots,\cos \theta_{2r}\},~\Lambda_{2}=\diag\{\sqrt{-1}\sin \theta_{0},\ldots,\sqrt{-1}\sin \theta_{2r}\}$, and $U$ is a unitary matrix coming from the {\rm SVD} of $S$ in \eqref{eq-S}. It is clear by construction that $\gamma_d\in \mathbf{H}_{d,n}\setminus \mathbf{H}_{d,n-1}$.

We thus obtain by a straightforward counting that
$2r(2r-5)+9$
gives a lower bound to the dimension of the moduli space of all required complex symmetric matrices $S$, which implies the dimension estimate of $\mathbf{H}_{d,n}$.
\end{proof}

\begin{remark}\label{lijin} 
Theorem~{\rm\ref{thm-Hd}} warrants that ${\mathbf H}_{d,n}$ is not empty so long as $n\geq d+2$. We will establish in the following proposition that ${\mathbf H}_{5,5}$ is empty and ${\mathbf H}_{4,4}/O(4;\mathbb{R})={\mathbf H}_{4,5}/O(5;\mathbb{R})$ is made up of a single point, so that Theorem~{\rm\ref{thm-Hd}} is optimal in general.
{\rm(}In fact, ${\mathbf H}_{3,3}/O(3;\mathbb{R})$ and ${\mathbf H}_{3,4}/O(4;\mathbb{R})$ are also empty, which  is part of the classification for $d\leq 3$ in Section $5$.
\end{remark}
\begin{prop}\label{Pr}{\rm\cite{zbMATH05590797}}
${\mathbf H}_{5,5}=\varnothing$. ${\mathbf H}_{4,4}/O(4;\mathbb{R})={\mathbf H}_{4,5}/O(5;\mathbb{R})$ is made up of a single point.
\end{prop}
\begin{proof}
We give a singular-value proof. When $n=d+1$, 1 is the singular value of the symmetric $n\times n$-matrix $S$ with multiplicity $d$. Let $\lambda\geq 0$ be the remaining singular value. Then there is a unitary matrix $U$ such that $S=U\,\diag(\lambda,1,1,\cdots,1)\,^{t}U$. So,
\begin{equation}\label{u}\overline{S}S=\overline{U}\;\diag(\lambda^2,1,1,\cdots, 1) \;^tU.\end{equation}


Expanding the right hand side of~\eqref{u} and utilizing that $U:=(u_{ij})$ is unitary, we derive
\begin{equation}\label{SS}
\overline{S}S=(t_{kl}),\quad\quad t_{kl}:=\delta_{kl}+(\lambda^2-1)\,\overline{u_{k0}}u_{l0},\quad 0\leq k,l\leq d.
\end{equation}

Now let $d=n=5$. It follows that $\lambda =1$ and $S$ is itself unitary. We can assume by Lemma~\ref{kill parameter lemma} below that the top three and bottom two of the anti-diagonals in the $6\times 6$ matrix $S$ are zero. Multiplying against different columns sets up a straightforward elimination process, bearing that $S$ is of rank 6, to let us end up with the $2\times 2$ anti-diagonal block form
$$
S=\antidiag(A, B, A).
$$
 But then $A$ and $B$ being unitary contradicts~\eqref{requirement}. This proves that ${\mathbf H}_{5,5}$ is empty.

We now assume $d=4$, continuing to denote the symmetric matrix now of size $5\times 5$ by $S$. If $\lambda=1$, then $S$ is unitary, a similar analysis as in the preceding case results in
\begin{equation}\label{singleton}
S:=\antidiag(1,-1,(-1)^{2},\ldots,(-1)^{d}).
\end{equation}
By the simple combinatorial identity $\sum_{i=0}^{4}(-1)^{i}\tbinom{4}{i}=0$, we know $S\notin \mathbf{S}_{4,5}\setminus\mathbf{S}_{4,4}$.
So in fact this constantly curved 2-sphere lies in a 3-quadric sitting in ${\mathbb C}P^4$, i.e., it belongs to $\mathbf{H}_{4,4}$.

Otherwise, $\lambda<1$ now.
Again, we can choose $S$ so that the top three and the bottom two anti-diagonals are zero, from which we see that the $(0,4)$-entry of $\overline{S}S$ is zero; in particular $\overline{u_{00}}u_{40}=0$. If, say, $u_{00}=0$, then~\eqref{SS} implies $t_{0j}=0$ for $j\neq 0$, which in turn implies that the first column of $S$ is unitarily orthogonal to other columns, etc.
It follows that a similar process of elimination as in the case $d=5$ returns us, if we put $S:=(s_{ij}), 0\leq i,j\leq 4 $, the values $s_{03}=-s_{12}=1$ and zero for all other entries. However, we also know that the Veronese curve $Z_d$ must satisfy~\eqref{requirement}, from which we deduce $s_{03}+\sqrt{6}s_{12}=0$. This contradiction shows the impossibility when $\lambda<1$. Thus, ${\mathbf H}_{4,4}/O(4;\mathbb{R})={\mathbf H}_{4,5}/O(5;\mathbb{R})$ is a singleton set, represented by~\eqref{singleton}.
\end{proof}

\begin{lemma}\label{kill parameter lemma}
For each $[S]\in \mathscr{S}_{d}/U(1)\times SU(2)$, we can find a representative $S_{0}:=(s_{k,l})\in [S]$, such that
$$s_{0,0}=s_{1,0}=s_{0,1}=s_{d,d}=s_{d-1,d}=s_{d,d-1}=0.$$
Moreover, if $d\geq 3$, $S_0$ can be chosen also satisfying
$$s_{1,1}=s_{0,2}=s_{2,0}=0.$$
\end{lemma}
\begin{proof} The first equality holds by a direct calculation from \eqref{requirement}. To verify the second,
consider the projection $\Pi_{1}$ of $\mathscr{S}_{d}$ into the first summand $\mathcal{V}^{d-2}$; see Proposition \ref{quadric contains rnc and decomposition}. Write
\begin{equation*}
\Pi_{1}(S_{0})=\sum_{j=0}^{2d-4}\alpha_{j}\;u^{2d-4-j}v^{j}.
\end{equation*}
Consider the induced representation of Lie algebra $\mathfrak{sl}(2,{\mathbb C})$ on $\mathscr{S}_{d}$ from that on $Sym_{d+1}(\mathbb{C})$, set $J_{3}:=\diag(-1,1)/2\in \mathfrak{sl}(2,{\mathbb C})$. The eigenvector space of $J_{3}$ in $\mathscr{S}_{d}$ corresponding to eigenvalue $d-2$ is of dimension $1$ because it is
\begin{equation*}
Span_{\mathbb{C}}\{\frac{1}{2}(e_{0}\otimes e_{2}+e_{2}\otimes e_{0}),~e_{1}\otimes e_{1}\}\cap~\mathscr{S}_{d}.
\end{equation*}
So we have $s_{1,1}=s_{0,2}=s_{2,0}=0$ if and only if $\alpha_{0}=0$. For $g=\begin{pmatrix}
                                                                     a & b \\
                                                                     -\bar{b} & \bar{a} \\
                                                                   \end{pmatrix}\in SU(2)
$, by simple calculation
\begin{equation*}
 \aligned
& \Pi_{1}(\varrho^{\frac{d}{2}}\otimes \varrho^{\frac{d}{2}}(g)S_{0})=\varrho^{d-2}(g)\circ\Pi_{1}(S_{0})\\
 &=\sum_{l=0}^{2d-4}\alpha_{l}\;(\bar{a}u-bv)^{2d-4-l}(\bar{b}u+av)^{l}
=:\sum_{l=0}^{2d-4}\widetilde{\alpha}_{l}\;u^{2d-4-l}v^{l},
\endaligned
\end{equation*}
where $\widetilde{\alpha}_{0}=\sum_{l=0}^{2d-4}\alpha_{l}\bar{a}^{2d-4-l}\bar{b}^{l}$. It is easy to be verified that $a,b\in \mathbb{C},~|a|^{2}+|b|^{2}=1$ can be chosen such that $\widetilde{\alpha}_{0}=0$.
\end{proof}

To conclude this section, we present a result about the existence of complex symmetric matrix $S\in\mathbf{S}_{d,n}$ taking $0$ as its minimal singular value with a given multiplicity. We point out that this is useful in the construction of constantly curved holomorphic 2-spheres in a singular hyperquadric, which will not be discussed in this paper.
\begin{lemma}\label{matrix-lemma1}
For any given $m\leq d$, there exists solutions to the equation
$$\sum_{j=0}^{m}s_{j,\,m-j}\;\sqrt{\tbinom{d}{j}\,\tbinom{d}{m-j}}=0,$$
such that $s_{j,\, m-j}=s_{m-j,\,j}\neq0$ for all $0\leq j\leq m$.

\end{lemma}

\begin{proof}
Similar to the proof of Lemma~\ref{matrix-lemma}, this follows from Lemma~\ref{elme-lemma1}, Remark~\ref{lemma1-remark}, Lemma~\ref{elem-lemma2} and Remark~\ref{lemma2-remark}.
\end{proof}

\begin{theorem}\label{thm-presin0}
For any given $q\geq 0$, there exists a constantly curved holomorphic $2$-sphere of degree $d$, by which the $d+1$ plane spanned takes $0$ as its minimal singular value with multiplicity $q$.
\end{theorem}
\begin{proof}
The proof is similar to that of Theorem~\ref{thm-Hd}, except that in the proof we use Lemma~\ref{matrix-lemma1} instead of Lemma~\ref{matrix-lemma}. In fact, now the required symmetric matrix can be defined by
$$S=\diag\Big\{\antidiag\{s_{0,\, m},~s_{1,\, m-1},\;\, \cdots,\;\,s_{m-1,\,1},\;\,s_{m,\,0}\},\,\,\underbrace{0,\cdots,0}_{q}\Big\},$$
where $m=d-q$ and $\{s_{0,\, m},~s_{1,\, m-1}, \;\,\cdots,s_{m-1,\,1},\;\,s_{m,\,0}\}$ is given in Lemma~\ref{matrix-lemma1}.
\end{proof}

Consequently, combining Theorem~\ref{thm-Hd} with Theorem~\ref{thm-presin0}, we may prescribe the multiplicities of both singular values $0$ and $1$.

\section{Classification of constantly curved holomorphic 2-spheres of degree $\leq 3$ in $\mathcal{Q}_{n-1}$}\label{sec-classify}
Constantly curved holomorphic 2-spheres of degree $d\leq 3$ in $\mathcal{Q}_{n-1}$ are discussed in this section. Since we need only consider the linearly full case, we deal only with $\mathcal{Q}_{n-1}$ with $d\leq n\leq 2d+1$. Classification of the moduli space
$$\mathbf{H}_{d,n}/O(n+1;\mathbb{R}),~~~1\leq d\leq 3,~~~d\leq n\leq 2d+1,$$
is obtained.

\subsection{Case of $d=1$.}
It is clear that there are no isotropic lines in $\mathbb{C}P^{n}$ when $n\leq 2$, so that $\mathbf{H}_{1,1}$ and $\mathbf{H}_{1,2}$ are empty. In $\mathbb{C}P^3$, the isotropic line can be determined, up to a real orthogonal transformation and reparametrization of $\mathbb{C}P^{1}$, by
\begin{equation*}
 \gamma:\mathbb{C}P^{1}\rightarrow \mathcal{Q}_{2},\quad\quad\quad\quad
 [u,v]\mapsto ~^{t}[~\frac{1}{\sqrt{2}}(u,v,\sqrt{-1}u,\sqrt{-1}v)]. 
\end{equation*}
Hence ${\mathbf H}_{1,3}/O(4;\mathbb{R})$ is a singleton set.

\subsection{Case of $d=2$.}
By a direct computation, it is easy to verify that
\begin{small}
$$\mathbf{S}_{2,5}/U(1)\times SU(2)\!\!=\!\!\Big\{\![\antidiag(\sigma,-\sigma,\sigma)]\Big|\sigma\in [0,1]\!\Big\},~~\mathbf{S}_{2,4}/U(1)\times SU(2)\!\!=\!\!\Big\{\![\antidiag(1,-1,1)]\!\Big\}.$$
\end{small}
Combining this with Theorem~\ref{simple classification}, we have the following.
\begin{theorem}
The  moduli space
$$\mathbf{H}_{2,2}/O(3;\mathbb{R})=\mathbf{H}_{2,3}/O(4;\mathbb{R})=\mathbf{H}_{2,4}/O(5;\mathbb{R})$$
is a singleton set, and
$$\mathbf{H}_{2,5}/O(6;\mathbb{R})\cong [0,1].$$
\end{theorem}

The constantly curved holomorphic 2-sphere of degree $2$ in $\mathcal{Q}_{1}\subset\mathcal{Q}_2\subset\mathcal{Q}_3$ can be parameterized as
$$\gamma: [u,v]\mapsto ~^{t}[\sqrt{-1}(u^{2}+v^{2}),\,u^{2}-v^{2},\,2uv,\,0,\,0].$$
The $1$-family of constantly curved holomorphic 2-spheres of degree $2$ in $\mathcal{Q}_{4}$ can be parametrized as
\begin{small}
$$\gamma_\sigma: [u,v]\mapsto ~^{t}[\sqrt{-1}\sigma(u^{2}+v^{2}),\,\sigma(u^{2}-v^{2}),\,2\sigma uv,\,\sqrt{1-\sigma^2} (u^{2}+v^{2}),\,\sqrt{\sigma^2-1}(u^{2}-v^{2}),\,2\sqrt{\sigma^2-1}uv].$$
\end{small}

To conclude the discussion of this case, we present a figure (Fig. 6.1) to show the distribution of constantly curved holomorphic 2-spheres of degree $2$ in $\mathcal{Q}_4$.
\begin{figure}[H]\label{Figure for deg 2.0}
\centering
\begin{tikzpicture}[scale=0.60]
\draw[->] (0,0) -- (xyz cs:x=5) node[anchor=north west] {$\sigma_{1}$ };;
\draw[->] (0,0) -- (xyz cs:y=5) node[anchor=south west] {$\sigma_{2}$ };;
\draw[->] (0,0) -- (xyz cs:z=5) node[anchor=north west] {$\sigma_{0}$ };
\draw[thin] (0,0)--(xyz cs:z=0,x=0,y=3)--(xyz cs:z=0,x=3,y=3);
\draw[thin] (xyz cs:z=0,x=3,y=3)--(0,0);
\draw[thick] (xyz cs:z=0,x=1.2,y=2)--(xyz cs:z=0,x=0,y=3);
\fill[fill=yellow!80!white] (xyz cs:z=0,x=0,y=3)--(xyz cs:z=0,x=1.2,y=2)--(0,0)--(xyz cs:z=0,x=0,y=3);
\fill[fill=yellow!80!white] (xyz cs:z=0,x=3,y=3)--(xyz cs:z=0,x=1.2,y=2)--(0,0)--(xyz cs:z=0,x=3,y=3);
\fill[fill=gray!40!white] (xyz cs:z=0,x=0,y=3)--(xyz cs:z=0,x=3,y=3)--(xyz cs:z=0,x=1.2,y=2)--(xyz cs:z=0,x=0,y=3);
\draw[thick,green] (xyz cs:z=0,x=3,y=3)--(xyz cs:z=0,x=1.2,y=2);
\draw[ultra thick, red] (xyz cs:z=0,x=1.2,y=2)--(0,0);
\draw (0,0) node[anchor=north] {(0,0,0)};
\draw (xyz cs:z=0,x=0,y=3) node[anchor=east] {(0,0,1)};
\draw (xyz cs:z=0,x=0,y=3) node[anchor=east] {(0,0,1)};
\draw (xyz cs:z=0,x=3,y=3) node[anchor=west] {(0,1,1)};
\draw (xyz cs:z=0,x=1.2,y=2) node[anchor=north west] {(1,1,1)};
\draw[thick, gray] (xyz cs:z=0,x=1.2,y=2)--(xyz cs:z=0,x=0,y=3);
\draw[thick, gray] (xyz cs:z=0,x=0,y=3)--(xyz cs:z=0,x=3,y=3);

\fill[fill=red] (0,0) circle (0.08cm);
\fill[fill=blue] (xyz cs:z=0,x=1.2,y=2) circle (0.12cm);
\fill[fill=green] (xyz cs:z=0,x=3,y=3) circle (0.08cm);
\fill[fill=gray!40!white] (xyz cs:z=0,x=0,y=3) circle (0.08cm);
\end{tikzpicture}
\caption{$\Theta_{2,2}=\mathbf{H}_{2,2}/O(3;\mathbb{R}), \,\Theta_{2,3}, \,\Theta_{2,4}, \,\Theta_{2,5}, \mathbf{H}_{2,5}/O(6;\mathbb{R})$}
\end{figure}
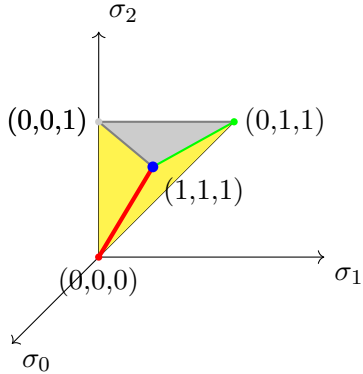
The tetrahedron in this figure 
represents $\Theta_{2,5}$, i.e., the orbit space of $G(3,6,\mathbb{C})/O(6;\mathbb{R})$ (see Theorem~\ref{moduli of plane}). 
$\mathbf{H}_{2,5}/O(6;\mathbb{R})$ is represented by the red line segment; it comprises, up to equivalence, of all planes whose intersection with $\mathcal{Q}_4$ contain constantly curved holomorphic 2-spheres of degree $2$. 
The blue endpoint of this segment represents $\Theta_{2,2}\cong G(3,3,\mathbb{C})/O(3;\mathbb{R})$; the constantly curved holomorphic 2-spheres lying in this plane are not linearly full in $\mathcal{Q}_{4}$. $\Theta_{2,3}\cong G(3,4,\mathbb{C})/O(4;\mathbb{R})$ is represented by the green line segment; it lies on the boundary of the grey triangle representing $\Theta_{2,4}\cong G(3,5,\mathbb{C})/O(5;\mathbb{R})$.

\subsection{Case of $d=3$.}\label{deg=3}
From Lemma \ref{kill parameter lemma}, it is easy to verify that
\begin{equation}\label{eq-Sform}
\mathscr{S}_{3}/U(1)\times SU(2)=\Big\{[\begin{pmatrix}
  0 & 0 & 0 & \frac{3}{2}y \\
  0 & 0& -\frac{1}{2}y & -\frac{\sqrt{3}}{2}z \\
  0& -\frac{1}{2}y & z & 0 \\
  \frac{3}{2}y & -\frac{\sqrt{3}}{2}z & 0 & 0 \\
\end{pmatrix}]\,\Big|\,y,\,z\in[0,+\infty)\Big\},
\end{equation}
where we have used the action of $U(1)$ and $SU(2)$ to modify the values of $y$ and $z$ such that they are nonnegative real numbers.
\begin{theorem}
$\mathbf{H}_{3,3}=\mathbf{H}_{3,4}=\varnothing$, $\mathbf{H}_{3,5}/O(6;\mathbb{R})$ is a singleton set, and

 {\bf (1)} $\mathbf{H}_{3,6}/O(7;\mathbb{R})$ is bijective to the closure of a $1/4$-circle, and,

{\bf (2)} $\mathbf{H}_{3,7}/O(8;\mathbb{R})$ is bijective to the closure of a $1/4$-disk.
\end{theorem}

\begin{proof}
Take a real symmetric matrix $S\in\mathscr{S}_{3}$ having the form given in \eqref{eq-Sform} and let $\lambda_1, \lambda_2,\lambda_3, \lambda_4$ be the eigenvalues of $S$. Their absolute values are just the singular values of $S$.
We find the characteristic polynomial of $S$ to be
\begin{equation}\label{charac polynomial}
\lambda^{4}+(-z)\lambda^{3}+(-\frac{3}{4}z^{2}-\frac{5}{2}y^{2})\lambda^{2}+(\frac{3}{4}z^{3}+\frac{9}{4}y^{2}z)\lambda+\frac{9}{16}y^{4}=0.
\end{equation}
The eigenvalues of $S$ are
\begin{equation}\label{eigenvalues}
 \aligned
 &\lambda_{1}=A-\frac{1}{2}[B+\frac{1}{2}C]^{\frac{1}{2}}-D,\quad\quad\quad&
 \lambda_{2}&=A+\frac{1}{2}[B+\frac{1}{2}C]^{\frac{1}{2}}-D,\\
& \lambda_{3}=A-\frac{1}{2}[B-\frac{1}{2}C]^{\frac{1}{2}}+D,\quad\quad\quad\quad&
 \lambda_{4}&=A+\frac{1}{2}[B-\frac{1}{2}C]^{\frac{1}{2}}+D,\\
\endaligned
\end{equation}
where $A=\frac{1}{4}z,~B=4y^{2}+2z^{2},~C=4z(y^{2}+\frac{z^{2}}{4})^{\frac{1}{2}},~D=\frac{1}{2}(y^{2}+\frac{z^{2}}{4})^{\frac{1}{2}}$. It is routine to check that $\lambda_{4}\geq -\lambda_{1}\geq \lambda_{2}\geq -\lambda_{3}\geq0$.
So the singular values of $S$ (in nondecreasing order) are given by $$\sigma_{1}=-\lambda_{3},~~~\sigma_{2}=\lambda_{2},~~~\sigma_{3}=-\lambda_{1},~~~\sigma_{4}=\lambda_{4}.$$
It follows that
\begin{equation}\label{injective}
z=\sigma_{4}-\sigma_{3}+\sigma_{2}-\sigma_{1},\quad
 9y^{4}/16=\sigma_{1}\sigma_{2}\sigma_{3}\sigma_{4},
\end{equation}
which implies the map
\begin{align}\label{singular values for deg 3}
 \phi:[0,\infty)\times[0,\infty)\rightarrow \mathbb{R}^{4},\quad\quad
(z,y)\mapsto (\sigma_{1},\sigma_{2},\sigma_{3},\sigma_{4})
\end{align}
is injective. This means $S(y,z)$ can be determined uniquely by its singular values, so that we have $\mathscr{S}_{3}/U(1)\times SU(2)\cong [0,\infty)\times [0,\infty)$.

Next, we consider the subsets $\mathbf{S}_{3,n}/U(1)\times SU(2)$, for which the following relations between singular values and $y,z$ are needed. They can be obtained from \eqref{eigenvalues} as follows.
\begin{align}\label{singular value distribution for deg 3}
\begin{split}
 (1)&~\text{If $yz\neq 0$, then $\sigma_{4}>\sigma_{3}>\sigma_{2}>\sigma_{1}>0$.}\\
 (2)&~\text{If $y\neq 0,~z=0$, then $\sigma_{4}=\sigma_{3}>\sigma_{2}=\sigma_{1}>0$.}\\
(3)&~\text{If $y=0,~z\neq 0$, then $\sigma_{4}>\sigma_{3}=\sigma_{2}>\sigma_{1}=0$.}\\
(4)&~\text{If $y=z=0$, then $\sigma_{1}=\cdots=\sigma_{4}=0$.}\\
\end{split}
\end{align}
A consequence is that $\sigma_4=\sigma_3=\sigma_2\neq0$ cannot happen, which implies $\mathbf{S}_{3,3}=\mathbf{S}_{3,4}=\varnothing$, and hence $\mathbf{H}_{3,3}=\mathbf{H}_{3,4}=\varnothing$.

For a matrix in $\mathbf{S}_{3,5}$, there must hold $\sigma_4=\sigma_3=1$. Then it follows from \eqref{singular value distribution for deg 3} that $z=0$ and $9y^4=16\sigma_1^2$. Substituting these and $\lambda=\sigma_1$ into \eqref{charac polynomial}, we solve to see that $\sigma_1=\sigma_2=\frac{1}{3}$ and $y=2/3$. This implies $\mathbf{H}_{3,5}/O(6;\mathbb{R})\cong\mathbf{S}_{3,5}/U(1)\times SU(2)$ is a singleton set.

To determine
\begin{small}
\begin{equation*}
\mathbf{S}_{3,6}/U(1)\times SU(2)\cong \{(y,z)\in\mathbb{R}^{2}_{\geq 0}|\sigma_{4}(y,z)=1\}\subset \mathbf{S}_{3,7}/U(1)\times SU(2)\cong\{(y,z)\in\mathbb{R}^{2}_{\geq 0}|\sigma_{4}(y,z)\leq 1\},
\end{equation*}
\end{small}
we need only note from \eqref{eigenvalues} that $\sigma_4(y,z)=1$ is a component $\eta$ of the following real algebraic curves of degree $4$
$$9y^4+12z^3+36y^2z-40y^2-12z^2-16z+16=0.$$
The part of the curve $\eta$ lying in the first quadrant is plotted below (Fig. 6.2), labeled by the red color; it represents the moduli space $\mathbf{H}_{3,6}/O(7;\mathbb{R})$.
In the figure, the blue area bounded by the curve $\eta$ and the coordinate axes gives the moduli space $\mathbf{H}_{3,7}/O(8;\mathbb{R})$.
\end{proof}
\begin{figure}[h]\label{Figure for deg 3}
\centering
\begin{tikzpicture}[scale=0.60]
        \begin{polaraxis}
 \fill[fill=blue!70!white,opacity=0.5] (canvas polar cs:angle=0, radius=3.12cm)--(canvas polar cs:angle=90, radius=2.10cm)--(canvas polar cs:angle=0, radius=0cm)--cycle;
 \addplot[fill=blue!70!white,opacity=0.5,domain=0:90,samples=360,smooth] (x,{sqrt(1+tan(x)^2)/(1/4+1/2*sqrt(4*tan(x)^2+2-2*sqrt(tan(x)^2+1/4))+1/2*sqrt(tan(x)^2+1/4))});
\addplot[red,domain=0:90,samples=360,smooth,ultra thick] (x,{sqrt(1+tan(x)^2)/(1/4+1/2*sqrt(4*tan(x)^2+2-2*sqrt(tan(x)^2+1/4))+1/2*sqrt(tan(x)^2+1/4))});
 \fill[fill=red] (canvas polar cs:angle=0, radius=3.1cm) circle (0.08cm) node[anchor=north east] {(1,0)};
\fill[fill=green] (canvas polar cs:angle=90, radius=2.09cm) circle (0.08cm) node[anchor=south west] {(0,\,$2/3$)};
\end{polaraxis}
 \end{tikzpicture}
\caption{$z-y$ plane}
\end{figure}
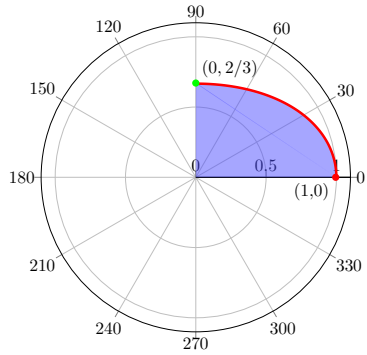


To conclude, we give an explicit example of constantly curved holomorphic 2-sphere of degree $3$ in $\mathcal{Q}_{6}$ to illustrate the constructive procedure given in Remark~\ref{construction}. 

\begin{example}
For convenience, we denote the matrix in \eqref{eq-Sform} by $S(\frac{1}{2},\frac{1}{3})$ with parameters $z=\frac{1}{2},~y=\frac{1}{3}$. Now
\begin{equation*}
S(\frac{1}{2},\frac{1}{3})=\begin{pmatrix}
                                        0 & 0 & 0 & \frac{1}{2} \\
                                         0& 0  & -\frac{1}{6} & -\frac{\sqrt{3}}{4} \\
                                        0 & -\frac{1}{6} & \frac{1}{2} & 0 \\
                                        \frac{1}{2} & -\frac{\sqrt{3}}{4} & 0 & 0 \\
                                      \end{pmatrix}.
\end{equation*}
Using \eqref{eigenvalues} and \eqref{singular values for deg 3}, we get the singular values of $S(\frac{1}{2},\frac{1}{3})$ to be
\begin{equation*}
 \sigma_{1}=\frac{\sqrt{19}}{12}-\frac{1}{3},\quad\sigma_{2}=\frac{1}{2},\quad
 \sigma_{3}=\frac{2}{3},\quad\sigma_{4}=\frac{\sqrt{19}}{12}+\frac{1}{3}.\\
\end{equation*}
Choose $a_j\in[0,\frac{\pi}{4}]$ such that $\cos2a_{j}=\sigma_{j},~1\leq j\leq 4$, and define $V_{\vec{\sigma}}$
as in \eqref{jm alpha0}. To the real symmetric matrix $S(\frac{1}{2},\frac{1}{3})$, a group of eigenvectors $v_1,v_2,v_3,v_4$ corresponding to the eigenvalues $-\sigma_{1},\,\sigma_2,\,-\sigma_3,\,\sigma_{4}$, respectively, is
\begin{equation*}
\aligned
&^{t}v_{1}=\begin{pmatrix}
        1 &
        \frac{4\sqrt{57}}{27}+\frac{\sqrt{3}}{54} &
        -\frac{\sqrt{57}}{27}+\frac{10\sqrt{3}}{27} &
        -\frac{\sqrt{19}}{6}+\frac{2}{3} &
      \end{pmatrix},~~^{t}v_{2}=\begin{pmatrix}
        1 &
       0 &
        -\frac{3\sqrt{3}}{2} &
        1
      \end{pmatrix},\\
      &^{t}v_{3}=\begin{pmatrix}
        1 &
       -\frac{14}{27}\sqrt{3} &
        -\frac{2\sqrt{3}}{27} &
        -\frac{4}{3}
      \end{pmatrix},~~^{t}v_{4}=\begin{pmatrix}
        1 &
        -\frac{4\sqrt{57}}{27}+\frac{\sqrt{3}}{54} &
        \frac{\sqrt{57}}{27}+\frac{10\sqrt{3}}{27} &
        \frac{\sqrt{19}}{6}+\frac{2}{3}
      \end{pmatrix}.
\endaligned
\end{equation*}
Set $U\in U(4)$ to be a unitary matrix whose rows from top to bottom are
   $$\sqrt{-1}\;^{t}v_{1}/|v_{1}|,\quad
   ^{t}v_{2}/|v_{2}|,\quad
    \sqrt{-1}\;^{t}v_{3}/|v_{3}|,\quad
   ^{t}v_{4}/|v_{4}|.$$
It satisfies $^{t}U\diag\{\sigma_{1},\sigma_{2},\sigma_{3},\sigma_{4}\}U=S(\frac{1}{2},\frac{1}{3}),$ so that $V_{\vec{\sigma}}U Z_{3}$ is a holomorphic $2$-sphere of degree $3$ and constant curvature in $\mathcal{Q}_{6}$.
\end{example}


\section{More geometry of minimal 2-spheres constructed by the SVD method}\label{sec-prob}


The main objective of this section is to discuss the two questions of Peng, Xu and Wang, raised in~\cite[Section 3,~p.~459]{PengWangXu} and mentioned in the introduction, as to construct nonhomogenous minimal as well as totally real minimal 2-spheres of constant curvature in $\mathcal{Q}_{n-1}$, such that they are also minimal in $\mathbb{C}P^{n}$.

As seen in the preceding sections, all constantly curved 2-spheres minimal in both $\mathcal{Q}_{n-1}$ and $\mathbb{C}P^{n}$ can be constructed by the SVD method, for which the norm of the second fundamental form $||B||$ is computed in this section. Then we show that $||B||$ is not generally constant, so that the generic 2-sphere constructed this way is not homogenous in $\mathcal{Q}_{n-1}$.


Recall the decomposition of $Sym_{d+1}(\mathbb{C})$ into $SU(2)$-invariant subspaces (see the discussion around \eqref{general projection} and \eqref{important subsapce of symm matrices}), and the definition of the set $\mathbf{H}_{d,n,p}$ in \eqref{rnc in hyperquadric}. The following is essentially Proposition \ref{quadric contains rnc and decomposition} with the promised proof.
\begin{prop}\label{summary theorem}
Let $d,n,p$ be three integers satisfying $1\leq d\leq n,~~0\leq p\leq [\frac{d}{2}]$. Consider a $2$-sphere given by $EZ_{d,p}$, where $E\in M(n+1,d+1)$ satisfying $E^{*}E=Id$. Let $S:=~^{t}\!EE$ and let the entries of $S$ be $s_{ij}$, where $0\leq i,j\leq d$, and $s_{ij}=s_{ji}$. Then the following are equivalent.

 {\bf (1)} The $2$-sphere $EZ_{d,p}$ lies in the hyperquadric $\mathcal{Q}_{n-1}$.


 {\bf (2)} $^{t}Z_{d,0}\, S\, Z_{d,l}=0,~0\leq l\leq 2p+1$.

{\bf (3)} $^{t}Z_{d,k}\, S\, Z_{d,l}=0,~0\leq k+l\leq 2p+1$.

{\bf (4)} The $2$-sphere $EZ_{d,k}$ lies in the hyperquadric $\mathcal{Q}_{n-1}$ for all $0\leq k\leq p$.

{\bf (5)} $\sum_{i,j=0}^{d}s_{ij}~j^{2k}\sqrt{\tbinom{d}{i}\tbinom{d}{j}}~z^{i+j-2k}=0,~0\leq k\leq p$.

 {\bf (6)} $S\in \ker \Pi_{0}\cap\cdots\cap\Pi_{p}$.
\end{prop}
 %

\begin{proof} (sketch)
Note that $(3)\Rightarrow (4)\Rightarrow (1)$ is clear. So, it suffices to show $(1)\Rightarrow (2)\Rightarrow(3),~(2)\Rightarrow (5)\Rightarrow (1)$, and $(5)\Leftrightarrow (6)$.

{\bf (1) $\Rightarrow$ (2)} Taking $\frac{\partial}{\partial z}$ on both sides of $^{t}Z_{d,p}\, S\,Z_{d,p}=0$, by \eqref{recursive formulas}, we get
\begin{equation}\label{induction step}
^{t}Z_{d,p}\, S\, Z_{d,p+1}=0.
\end{equation}
Taking $\frac{\partial}{\partial \bar{z}}$ on both sides of \eqref{induction step}, by \eqref{facts of veronese sequence} we see $^{t}Z_{d,p-1}\, S\, Z_{d,p+1}=0$.
Then taking $\frac{\partial}{\partial z}$ again yields $^{t}Z_{d,p-1}\, S\cdot Z_{d,p+2}=0$. 
Continuing the process $p$ times, we arrive at
\begin{equation}\label{final get}
^{t}Z_{d,0}\, S\, Z_{d,2p+1}=0.
\end{equation}
Then taking $\frac{\partial}{\partial \bar{z}}$ on \eqref{final get} for $2p+1$ times, we get (3) by \eqref{facts of veronese sequence}.

{\bf (2) $\Rightarrow$ (3)} Assuming $(2)$, then taking $\frac{\partial}{\partial z}$ on both sides, we get
$^{t}Z_{d,1}\, S\, Z_{d,l}=0,$
where $0\leq l\leq 2p$. Then by induction, we obtain (3).

{\bf (2) $\Rightarrow $ (5)} From $^{t}Z_{d,0}\, S\, Z_{d,l}=0,~0\leq l\leq 2p+1$, we deduce
\begin{equation*}
\sum_{i,j=0}^{d}s_{ij}~j(j-1)\cdots(j-l+1)\sqrt{\tbinom{d}{i}\tbinom{d}{j}}~z^{i+j-l}=0,~~~0\leq l\leq 2p+1.
\end{equation*}
Then expand
$j(j-1)\cdots(j-l+1)$
and note that the coefficients before $j^{k},~0\leq k\leq l$, are all non-zero.

{\bf (5) $\Rightarrow$( 1)} We use induction to prove it. The claim is right for $p=0$. Assume it is right for $0\leq k\leq p-1$. Then by the induction assumption, $(5)$ implies that
$^{t}Z_{d,0}\, S \, Z_{d,l}=0,$
for $0\leq l\leq 2p$. By an argument similar to that in $(3)\Rightarrow(4)$, we get $^{t}Z_{d,p}\, S \, Z_{d,p}=0$, which is equivalent to $(1)$.

{\bf (5) $\Leftrightarrow$ (6)} Consider the induced action of $\mathfrak{su}(2)$ on $Sym_{d+1}(\mathbb{C})$ and extend it to the action of $\mathfrak{sl}(2;\mathbb{C})$. For $t:=i+j$ fixed, where $0\leq t\leq 2d$, i.e., considering the entries on each anti-diagonal, we claim that the system of linear equations in (5) are linearly independent. We then count the multiplicity of the eigenvalues of $J_{3}:=\diag(-1,1)/2\in \mathfrak{sl}(2;\mathbb{C})$.


To prove the claim,
 if $t=2m$, set $a_{ij}=s_{ij}\sqrt{\tbinom{d}{i}\tbinom{d}{j}}$. Consider the following system of linear equations
  $$
 2\sum_{0\leq i\leq m-1}a_{i\,2m-i}+a_{m\, m}=0,\quad
 2\sum_{0\leq i\leq m-1}a_{i\, 2m-i}~[i^{2k}+(2m-i)^{2k}]+a_{m\, m}~m^{2k}=0,
$$
for $1\leq k\leq m$. We need only show that the associated $(m+1)\times(m+1)$ coefficient matrix of the unknowns $a_{ij}$ is nonsingular, which follows from the observation that it is a product of a lower triangular and a Vandermonde matrix. 

If $t=2m+1$, a similar argument to that in the preceding case takes care of the following system of linear equations,

$
 \sum_{0\leq i\leq m}a_{i\, 2m+1-i}=0,\quad \quad
 \sum_{0\leq i\leq m}a_{i\, 2m+1-i}\,[i^{2k}+(2m+1-i)^{2k}]=0,~1\leq k\leq m.
$

\end{proof}
For minimal 2-sphere $EZ_{d,p}\in{\mathbf H}_{d,n,p}$, 
we have 
$ds^{2}=\lambda^{2}dzd\bar{z}$ (see Subsections~\ref{2.2} and~\ref{2.4}),
where $\lambda=\sqrt{d+2p(d-p)}/(1+|z|^{2})$, and
\begin{equation*}
X=\frac{1}{\lambda}\frac{EZ_{d,p+1}}{|Z_{d,p}|},~~Y=-\frac{|Z_{d,p}|}{\lambda |Z_{d,p-1}|^{2}}EZ_{d,p-1}.
\end{equation*}
Recall Subsection~\ref{2.2}, the Gaussian curvature $K_{d,p}$ is $4/(d+2p(d-p))$ and the K\"ahler angle $\theta_{d,p} \in [0,\pi]$ satisfies $\cos\theta_{d,p}=(d-2p)/(2p(d-p)+d)$. From item (5) in Theorem \ref{summary theorem}, we know $Y$ lies in the hyperquadric $\mathcal{Q}_{n-1}$, so $\tau_{Y}=|^{t}Y\,Y|=0$. Since this 2-sphere is minimal in both $\mathcal{Q}_{n-1}$ and $\mathbb{C}P^{n}$, Theorem~\ref{minimal in cpn and qn-1} implies $\tau_{XY}=|^{t}\!X\,Y|=0$. It follows from \eqref{formula of norm of second fundamental form} that the norm of the second fundamental form satisfies
\begin{equation}
||B||^{2}=-2K_{d,p}+2+6\cos^{2}\theta_{d,p}-4\tau^{2}_{X}.
\end{equation}
Since for $Z_{d,p}$, the Gaussian curvature $K_{d,p}$ and the K\"ahler angle $\theta_{d,p}$ are constant, $||B||^{2}$ is constant if and only if $\tau_{X}$ is constant. By a direct computation,
  \begin{equation*}
\tau_{X}=|^{t}\!X\,X|=\frac{(d-p)!}{(d+2p(d-p))\cdot d!p!}\frac{|^{t}Z_{d,p+1}SZ_{d,p+1}|}{(1+|z|^{2})^{d-2p-2}},
\end{equation*}
where $S=~^{t}\!EE$. Following the argument of Theorem \ref{summary theorem} and by \eqref{local sections equivalent}, we derive
\begin{align}\label{final computations}
\begin{split}
|^{t}Z_{d,p+1}SZ_{d,p+1}|&=|^{t}Z_{d,0}S\frac{\partial^{2p+2}}{\partial z^{2p+2}}Z_{d,0}|,\\
 &=|\sum_{i,j=0}^{d}s_{ij}~\sqrt{\tbinom{d}{i}\tbinom{d}{j}}z^{i+j-(2p+2)}j(j-1)\cdots(j-2p-1)|.
\end{split}
\end{align}
\begin{theorem}\label{non homogenous thm}
 Suppose $EZ_{d,p}\in\mathbf{H}_{d,n,p}$ is a linearly full minimal $2$-sphere. If we set $S:=~^{t}\!EE\in \mathbf{S}_{d,n,p}$, then we have the following.

 {\bf (1)}  If $1\leq d\leq 2$, then $||B||^{2}$ is constant.

 {\bf (2)} If $p=[\frac{d}{2}]$, then $S=0$, $||B||^{2}$ is constant, $2d+1= n$, and
 $\mathbf{H}_{d,2d+1,[\frac{d}{2}]}/O(2d+2;\mathbb{R})$
is a singleton set, given by $[V_{\vec{0}}Z_{d,[\frac{d}{2}]}]$, where $\vec{0}\in\mathbb{R}^{d+1}$ is the zero vector.  

 {\bf (3)} If $d$ is even and $p=[\frac{d}{2}]-1$, then $||B||^{2}$ is constant.

 {\bf (4)} If $d\geq3$, $0\leq p<[\frac{d}{2}]$, and $p\neq [\frac{d}{2}]-1$ when $d$ is even, then $||B||^{2}$ is constant, if and only if, $EZ_{d,p+1}\in \mathbf{H}_{d, n, p+1}$.


\end{theorem}
\begin{proof}
 {\bf  (1)} By a direct computation using \eqref{final computations}.

 {\bf (2)} If $p=[\frac{d}{2}]$, then by taking conjugation, it follows from Proposition~\ref{summary theorem} that $EZ_{d,q}$ lies in $\mathbf{H}_{d,n,q}$ for any $0\leq q\leq d$. We conclude $S=0$, which implies all its singular values vanish.  Combining this with the linearly full assumption, we have $2d+1= n$ by Theorem~\ref{coro-dim}. Substituting $S=0$ into \eqref{final computations}, we have $|\tau_X|=0$ and hence $||B||^2$ is a constant.  That the moduli space $\mathbf{H}_{d,2d+1,[\frac{d}{2}]}/O(2d+2;\mathbb{R})$ is a singleton set follows from $\mathbf{S}_{d,2d+1,[\frac{d}{2}]}=\{[0]\}$ by using Theorem \ref{simple classification}.


{\bf (3)}  If $d$ is even, and $p=\frac{d}{2}-1$, then the only possible nonvanishing entries of $S$ are on the main anti-diagonal, i.e. $i+j=d$. So, $(1+|z|^{2})^{d-2p-2}=1$ and \eqref{final computations} is constant, which implies $|\tau_{X}|$ is constant and then $||B||^2$ is constant.

{\bf (4)} In this case, we need only show that $\tau_{X}$ is constant if and only if it is zero. For convenience, denote the polynomial $^{t}Z_{d,0}\,S\,\frac{\partial^{2p+2}}{\partial z^{2p+2}}Z_{d,0}$ in \eqref{final computations} by $g(z)$. This is a polynomial of $z$. Note that $\tau_{X}$ is zero if and only if $g(z)=0$.

If $\tau_{X}$ is constant, then we have
\begin{equation*}
|g(z)|^{2}/(1+|z|^{2})^{2d-4p-4}=c,
\end{equation*}
for some constant $c$. So $c(1+|z|^{2})^{2d-4p-4}=g(z)\cdot \overline{g(z)}$. If $d\geq 3,~0\leq p<[\frac{d}{2}]$, and $p\neq \frac{d}{2}-1$ when $d$ is even, then $2d-4p-4>0$. If $c\neq 0$, then
$|g(z)|^{2}$ is divided by $1+z\bar{z}$. It is well known that $\mathbb{C}[z,\bar{z}]$ is a unique factorization domain. 
Since $1+z\bar{z}$ is irreducible in $\mathbb{C}[z,\bar{z}]$, by unique factorization, either $g(z)$ or $\overline{g(z)}$ is divisible by $1+|z|^{2}$, say, $g(z)$. It leads to a contradiction by a degree count of $\bar{z}$ in $1+z\bar{z}$ and $g(z)$. So $c=0$, which implies $|\tau_X|^2=0$.

\end{proof}

Now we can answer the question raised by Peng, Wang and Xu, mentioned in {\bf Problem 1} said in the introduction as follows.
\begin{coro}
Let $d,n,p$ be three integers satisfying $3\leq d\leq n,~0\leq p<[\frac{d}{2}]$, and $p\neq [\frac{d}{2}]-1$ when $d$ is even. Then the generic minimal $2$-sphere $EZ_{d,p}$ in $\mathbf{H}_{d,n,p}$ is not homogenous.
\end{coro}

For instance, when $n=2d+1$, by item (4) of the preceding theorem, the corollary follows from the fact that $\mathbf{S}_{d,n,p+1}=\mathbf{S}_{d,n,p}\cap\ker\Pi_{p+1}$ is a proper subset of $\mathbf{S}_{d,n, p}$ with codimenison given by
\begin{equation}\label{difference}
\dim\mathbf{S}_{d,n,p}-\dim \mathbf{S}_{d,n,p+1}=4d-8p-6.
\end{equation}

Note that the constantly curved holomorphic 2-spheres of degree $3$ constructed in Section \ref{deg=3} are not homogeneous, if the parameters $(y,z)$ corresponding to them are not equal to $(0,0)$. 

More generally, observe that when $p=0$, item (4) of Theorem~\ref{non homogenous thm} says that the holomorphic 2-sphere of degree $d$ of constant curvature assumes constant $||B||^2$ if and only if the tangent developable surface ${\mathcal D}$ of the Veronese curve lies in the quadric defined by the symmetric matrix $S$. In Eisenbud~\cite{GreenConjecture}, regarding Green's conjecture, it is pointed out that if we let $z_0,\cdots,z_d$ be the homogeneous coordinates of ${\mathbb C}P^d$, $w_i=z_i/\binom{d}{i},0\leq i\leq d,$ and consider the $2\times d$ matrix
$$
\begin{pmatrix}w_0&w_1&\cdots&w_{d-1}\\w_1&w_2&\cdots&w_d\end{pmatrix},
$$
for which we let $\Delta_{a,b}, 0\leq a,b\leq d-3,$ be the quadric polynomials obtained by taking the determinant of the minor associated with the columns $a$ and $b$, then the quadratic equations of ${\mathcal D}$ are given by
$$
\Gamma_{a,b}:=\Delta_{a+2,b}-2\Delta_{a+1,b+1}+\Delta_{a,b+2}=0,\quad 0\leq a,b\leq d-3.
$$
Moreover, $\Gamma_{a,b}$ generate the ideal of ${\mathcal D}$ when $d\geq 6$.
The linear span of these $\Gamma_{a,b}$, of real dimension $2\binom{d-2}{2}=d^2-5d+6$,
constitutes the space of symmetric matrices containing ${\mathcal D}$, in agreement with $\dim\mathbf{S}_{d,n,1}$ and made explicit of the dimension given in~\eqref{difference} when $n=2d+1$ (note that $\dim \mathbf{S}_{d,2d+1,0}=d^2-d$).


Meanwhile, we present the following corollary of Theorem~\ref{non homogenous thm}, which relates to {\bf Problem 2} mentioned in the introduction. Note that the minimal surface is totally real if and only if its K\"ahler angle is $\pi/2$. For the Veronese maps, the only totally real case is $Z_{d,\frac{d}{2}}$, where $d$ is even; see the formula \eqref{eq-kahler} for $\cos\theta_{d,p}$.
\begin{coro}\label{cor-uniqueness}
A totally real linearly full minimal $2$-sphere of constant curvature $8/(d^2+2d)$ in $\mathcal{Q}_{n-1}$, such that it is also minimal in $\mathbb{C}P^{n}$, exists only when $d$ is even and $n=2d+1$. Moreover, it is unique up to isometric transformations of $\mathcal{Q}_{n-1}$.
\end{coro}


To conclude the paper, we propose two interesting questions for further study:
 
\textbf{Question 1}: In view of Remark \ref{d,d+2rmk} and Proposition \ref{Pr}, $\mathbf{H}_{d,d}$ and $\mathbf{H}_{d,d+1}$ are peculiar due to the relatively small degrees engaged. Can we understand their structures better? 
\\

\textbf{Question 2}:
Our work requires that the $2$-spheres be minimal in both the hyperquadric and ${\mathbb C}P^{n}$, while in \cite[p. 1022, (2)]{JiaoLiS^2inQn}, Jiao and Li found a constantly curved $2$-sphere which is only minimal in the hyperquadric. How to construct such $2$-spheres, minimal only in the hyperquadric, in a systematic way?
\\

\textbf{Acknowledgement}:
The second author is supported by NSFC No.11601513, the Fundamental Research Funds for Central Universities, and the CSC scholarship while visiting Washington University in St. Louis. The third author is supported by NSFC No.11871450 and acknowledges the support from the UCAS joint PhD Training Program. He would also thank Professor Xiaoxiang Jiao for guidance and suggestions. They both would thank the first author and the Department of Mathematics and Statistics at Washington University for warm hospitality during their visit.



\noindent Department of Mathematics and Statistics, Washington University, St. Louis, MO63130.\\
Department of Mathematics, China University of Mining and Technology (Beijing), Beijing 100083, China.\\
School of Mathematical Sciences, University of Chinese Academy of Sciences, Beijing 100049, China.\\
\noindent E-mail: chi@wustl.edu;~~~\phantom{,,}xiezhenxiao@cumtb.edu.cn;~~~\phantom{,,}xuyan2014@mails.ucas.ac.cn.

\end{document}